\documentclass[12pt,reqno]{amsart}
\usepackage{amsmath, amssymb, amsthm}
\usepackage{url}
\usepackage[breaklinks]{hyperref}
\usepackage{tikz}
\renewcommand{\Im}{\operatorname{Im}}

\renewcommand{\Re}{\operatorname{Re}}
\renewcommand{\Im}{\operatorname{Im}}

\renewcommand{\k}{\kappa}

\renewcommand{\(}{\left\(}
\renewcommand{\)}{\right\)}
\renewcommand{\[}{\left\[}
\renewcommand{\]}{\right\]}

\numberwithin{equation}{section}
 \theoremstyle{plain}
\newtheorem{theorem}{Theorem}[section]
\newtheorem{lemma}[theorem]{Lemma}
\newtheorem{remark}[]{Remark}

\newtheorem{corollary}[theorem]{Corollary}
\newtheorem{proposition}[theorem]{Proposition}

   \makeatletter
\def\proof{\@ifnextchar[{\@oproof}{\@nproof}}
\def\@oproof[#1][#2]{\trivlist\item[\hskip\labelsep\textit{#2 Proof of\
#1.}~]\ignorespaces}
\def\@nproof{\trivlist\item[\hskip\labelsep\textit{Proof.}~]\ignorespaces}
\makeatother

\begin{document}
\title[Rademacher-type exact formula for $r$-colored $\ell$-regular partitions]{Rademacher-type exact formula and higher order Tur\'{a}n inequalities for $r$-colored $\ell$-regular partitions}

\author{Archit Agarwal}
\address{Archit Agarwal\\ Department of Mathematics \\
Indian Institute of Technology Indore \\
Simrol,  Indore,  Madhya Pradesh 453552, India.} 
\email{archit.agrw@gmail.com,   phd2001241002@iiti.ac.in }

\author{Meghali Garg}
\address{Meghali Garg \\ Department of Mathematics \\
Indian Institute of Technology Indore \\
Simrol,  Indore,  Madhya Pradesh 453552, India.} 
\email{meghaligarg.2216@gmail.com,   phd2001241005@iiti.ac.in}

 \author{Bibekananda Maji}
\address{Bibekananda Maji\\ Department of Mathematics \\
Indian Institute of Technology Indore \\
Simrol,  Indore,  Madhya Pradesh 453552, India.} 
\email{bibek10iitb@gmail.com,  bmaji@iiti.ac.in}

\maketitle
\begin{abstract}
In 1937, Rademacher refined the circle method of Hardy and Ramanujan to derive an exact convergent series for the partition function $p(n)$.  In 1942,  
Hua derived an exact formula for the distinct part partition function, and in 1971, Hagis generalized this result to the case of $\ell$-regular partitions. More recently, Iskander, Jain, and Talvola established a Rademacher-type exact formula for the $r$-colored partition function. 
In this paper, we employ the circle method to obtain a Rademacher-type exact formula for $r$-colored $\ell$-regular partitions for any $r \in \mathbb{N}$ and $\ell \geq 2$. As an application, we derive higher order Tur\'{a}n  inequalities for the $r$-colored $\ell$-regular partition function using a result of Griffin, Ono, Rolen, and Zagier. Furthermore, as additional consequences, we establish Rademacher-type exact formulas and higher order Tur\'{a}n inequalities for the $r$-colored distinct part partition function and for the sum of minimal excludants over ordinary partitions and overpartitions.
\end{abstract}

\section{Introduction}

In the theory of partitions,  determining an exact value for the partition function $p(n)$ for large $n$ has been considered as one of the challenging problems for quite some time until
Hardy and Ramanujan \cite{HR1918} made a breakthrough in 1918 by developing the circle method and obtaining an asymptotic formula for $p(n)$ which enabled them to find values of partition function for large $n$.  
Mainly,  they showed that,  as $n \rightarrow \infty$,  
$$
p(n) = \frac{e^{\pi \sqrt{\frac{2 \lambda_n}{3}} }}
{4\sqrt{3 }\lambda_n}
\left( 1 - \frac{1}{\pi \sqrt{\frac{2\lambda_n}{3}} } \right)
+ O\!\left( \frac{e^{\pi \sqrt{\frac{2 }{3}} \sqrt{n/2}}}{n} \right),$$
where $\lambda_n= n - 1/24$. 
In particular,  one can see that,  
\begin{align*}
 p(n) \thicksim \frac{1}{4n \sqrt{3}} e^{ \pi \sqrt{ \frac{2 n}{3}}} \quad \mathrm{as} \quad n \rightarrow \infty .
\end{align*}
The circle method has emerged as one of the most powerful and indispensable techniques in the field of analytic number theory.   This technique was further improved by Rademacher \cite{HRad1937, HRad1943},  who derived the following exact formula for $p(n)$: 
\begin{align}\label{Rademacher formula}
p(n) = 2\pi \left(\frac{1}{6\sqrt{\frac{2}{3}\left(n-\frac{1}{24}\right)}}\right)^{\frac{3}{2}} \sum_{k=1}^{\infty} \frac{A_{k}(n)}{k} I_{\frac{3}{2}}\left(\frac{\pi}{k}\sqrt{\frac{2}{3}\left(n-\frac{1}{24}\right)} \right),
\end{align}
where 
\begin{align*}
A_{k}(n) = \sum_{\substack{h ~\mathrm{mod}~k \\ \gcd(h,k)=1}} e^{i\pi s(h,k) -2\pi i n \frac{h}{k}},
\end{align*}
and $I_{\nu}$ denotes the modified Bessel function of the first kind and $s(h,k)$ denotes  the Dedekind sum defined as 
\begin{align}\label{dedekind_sum}
s(h,k):= \sum_{r=1}^{k-1}\frac{r}{k} \left( \frac{hr}{k} - \left\lfloor \frac{hr}{k} \right\rfloor - \frac{1}{2} \right).
\end{align}
A detailed discussion on Radmacher's proof of \eqref{Rademacher formula} can be found in  \cite{And1998, Apostal1990, HRad1973}.  
Over the years,  Hardy-Ramanujan-Rademacher circle method has been used extensively to derive Rademacher-type exact formula for various restricted partition functions.   For any $r \in \mathbb{N}$, the $r$-colored ordinary partition function $p^{(r)}(n)$ is a natural generalization of the partition function $p(n)$.  Numerous mathematicians have studied the arithmetic properties of $p^{(r)}(n)$,  see \cite{BCCDGS23, Dicks23}.  

Recently,  for $1 \leq r \leq 24$,  Pribitkin and Williams \cite{rcolour} obtained an exact formula for $r$-colored partitions using the duality between modular forms of weight $\frac{-r}{2}$ and $2 + \frac{r}{2}$. Motivated by their work, we aim to establish a Rademacher-type exact formula for the $r$-colored partition function for \emph{any} $r \in \mathbb{N}$. 
 Mainly, we prove the following result.  
%\begin{theorem}\label{r_colour_partitions}
For $ r \in \mathbb{N}$ and $n > \frac{r}{24}$,  we have 
\begin{align*}
p^{(r)}(n) = 2\pi \sum_{k=1}^{\infty} \frac{1}{k}  \sum_{m=0}^{\lfloor{\frac{r}{24}}\rfloor}  p^{(r)}(m) A (n) I_{1+\frac{r}{2}} \left(\frac{4\pi}{k} \sqrt{\left(\frac{r}{24} -m\right)\left(n-\frac{r}{24}\right)}\right) \left(\frac{\frac{r}{24} -m}{n-\frac{r}{24}}\right)^{\frac{1}{2}+\frac{r}{4}},
\end{align*}
where 
\begin{align*}
A(n) := A_{k,r,m}(n) := \sum_{\substack{h~\mathrm{mod}~ k\\ \gcd(h,k)=1}} e^{ir\pi s(h,k) +\frac{2\pi i}{k}(mh'-nh)} ,  
\end{align*}
with $hh' \equiv -1(\mathrm{mod}~k)$ and $s(h,k)$ is defined in \eqref{dedekind_sum}.
%\end{theorem}
However,  later we came to know that the above formula has already been obtained by Iskander,  Jain and Talvola \cite[Theorem 1.1]{IJT20} for \emph{any real} $r>0$.

In 1942,  Hua \cite{LKHUA} obtained a Rademacher-type exact formula for partitions into distinct parts.  Let $p_d(n)$ be the distinct part partition function.   Hua showed that 
\begin{align}\label{exact_formula_pd(n)}
p_d(n) = \frac{1}{\sqrt{2}}\sum_{\substack{k=1\\k~\mathrm{odd}}}^{\infty}E_{k}(n) \frac{\mathrm{d}}{\mathrm{d}n}J_0\left( \frac{i\pi}{k\sqrt{3}}\sqrt{n+\frac{1}{24}}  \right),
\end{align}
where 
\begin{align*}
E_{k}(n)= \sum_{h \bmod k} e^{\pi i (s(h, k)-s(2h, k)) - 2\pi i n\frac{h}{k}},
\end{align*} 
and $J_0$ is the Bessel function of  the first kind.  

In this paper,  we extend Hua's result by establishing a Rademacher-type exact formula for $r$-colored distinct part partition function,  which we denote by $p_d^{(r)}(n)$.  

Moreover, for any positive integer $\ell \geq 2$,  we denote $\ell$-regular partitions of $n$ as $b_{\ell}(n)$ that counts the number of partitions of $n$ where parts are not divisible by $\ell$.  Many mathematicians studied Ramanujan-type congruence properties of $b_{\ell}(n)$ for different values of $\ell$.   Curious readers can see \cite{OB16, CW14,  CG2013, DP09,  Penniston} and references therein.   A Rademacher-type exact formula for $\ell$-regular partitions was established by Hagis \cite{Hagis} in 1971.  One of the main aims of this manuscript is to generalize the result of Hagis by deriving an exact formula for $r$-colored $\ell$-regular partition function $b_\ell^{(r)}(n)$,  whose generating function is given by 
\begin{align}\label{generating_b_l^{(r)}}
\sum_{n=0}^{\infty} b_\ell^{(r)}(n)q^n = \frac{(q^\ell; q^\ell)_\infty^r}{(q; q)_\infty^r}=  \prod_{n=1}^\infty \left( \frac{1 - q^{\ell n}}{1- q^n} \right)^r,  
\end{align}
where $(A;B)_\infty := \prod_{n=0}^\infty ( 1 - A B^n)$ for $|B|<1$ and $A \in \mathbb{C}$.

\subsection{Higher order Tur\'{a}n  inequalities}
In recent years, considerable attention has been devoted to the study of Tur\'{a}n inequalities and their higher order generalizations for combinatorial and number-theoretic sequences. A sequence $\{t(n)\}$ of real numbers is said to be log-concave if it satisfies the classical Tur\'{a}n inequality:
\begin{align*}
t(n)^2 \geq t(n-1)t(n+1),\quad \text{for}~ \text{all} \quad n \geq 1. 
\end{align*}
The investigation of log-concavity and Turán inequalities is deeply intertwined with the theory of real entire functions in the Laguerre–Pólya class and with analytic aspects related to the Riemann Hypothesis \cite{CNV86, Dimitrov, GORZ19, szego}.
These properties also frequently arise in combinatorics, where many classical sequences such as the binomial coefficients, Stirling numbers, and Bessel numbers are known to exhibit log-concavity \cite{S89}. For the partition function $p(n)$, Nicolas \cite{Nicolas78} first proved that
\begin{align*}
 p(n)^2 > p(n-1)p(n+1),
\end{align*}
for all $n > 25$,  which was later established again by DeSalvo and Pak \cite{DeSalvo2015} using Lehmer’s refinement of the error term in  the formula for $p(n)$.  Beyond classical log-concavity, the notion of higher order Tur\'{a}n inequalities offer a broader framework for studying real-rootedness phenomena via Jensen polynomials. For a real sequence $\{t(n)\}$, the Jensen polynomial of degree $d$ and shift $n$ is defined by
\begin{align*}
J_t^{d,n}(X) := \sum_{i=0}^{d} \binom{d}{i} t(n+i)\, X^i.
\end{align*}
The sequence $\{t(n)\}$ is said to satisfy the degree $d$ Tur\'{a}n inequality at $n$ if the polynomial $J_t^{d,n-1}(X)$ is hyperbolic, i.e., all its roots are real.
The study of Jensen polynomials has proven remarkably powerful in understanding the asymptotic behavior and analytic structure of arithmetic functions. For instance, Chen, Jia, and Wang \cite{CJW2019} established the hyperbolicity of the cubic Jensen polynomial $J_p^{3,n-1}(X)$ associated to the partition function $p(n)$ for all $n \geq 94$ and further conjectured that for every integer $d \geq 1$, there exists an integer $N_p(d)$ such that $J_p^{d,n-1}(X)$ is hyperbolic for all $n \geq N_p(d)$.
This conjecture was subsequently proved by Griffin, Ono, Rolen, and Zagier \cite[Theorem 5]{GORZ19}.  They not only proved the result for the partition function but also established the hyperbolicity of Jensen polynomials associated with the Fourier coefficients of weakly holomorphic modular forms on the full modular group $SL_2(\mathbb{Z})$.
Their result revealed a striking connection between Jensen polynomials and the Hermite polynomials $H_d(X)$.
They proved that, under mild analytic assumptions on a positive sequence $\{t(n)\}$,  properly scaled Jensen polynomials converge to Hermite polynomials as $n \to \infty$. 
%Namely, they proved the following beautiful result.  
\begin{theorem}\label{theorem of Griffin ono rolen and zagier} \cite[Theorem 3 and 8]{GORZ19}
Let $\{t(n)\}, \{A(n)\}, \{\delta(n)\}$ be sequences of positive real numbers, with $\delta(n)$ tending to $0$.  
For integers $j \geq 0$,  $d \geq 1$, suppose that there are real numbers $g_3(n), g_4(n), \ldots, g_d(n)$, for which 
\begin{align*}
\log \left( \frac{t(n+j)}{t(n)} \right)
= A(n)j - \delta(n)^2 j^2 + \sum_{i=3}^{d} g_i(n) j^i + o(\delta(n)^d),
\quad \text{as } n \to \infty,
\end{align*}
with $g_i(n) = o(\delta(n)^i)$ for each $3 \leq i \leq d$.  
Then we have  
\begin{align*}
\lim_{n \to \infty} \left( \frac{\delta(n)^{-d}}{t(n)}
J_{t}^{d,n} \!\left( \frac{\delta(n)X - 1}{\exp(A(n))} \right) \right)
= H_d(X).
\end{align*}
\end{theorem}
Since Hermite polynomials have distinct real zeros, and real-rootedness is preserved under linear transformations, this implies that Jensen polynomials have distinct real roots as well and consequently, the higher order Tur\'{a}n inequalities hold for large $n$.  Over time,  many mathematicians have investigated log-concavity and higher order Tur\'{a}n inequalities for various partition functions,   interested readers can see \cite{AGM25,  ChenTalk2010,  CJW2019,   Dimitrov,  LW19,  OPR22,  Pandey24}.  In 2021,  Crig and Pun \cite{CP21} examined the higher order Tur\'{a}n inequalities associated with the $\ell$-regular partition function.  In 2024,  Dong and Ji \cite{DJ24} also studied the same for the partition function into distinct parts.
%and more recently, the authors \cite{AGM25} explored log-concavity and higher order Tur\'{a}n inequalities for the cubic overpartition function.} 
In the present paper, we utilize the above result of Griffin, Ono, Rolen and Zagier to obtain higher order Tur\'{a}n inequalities for  $r$-colored $\ell$-regular partitions.

 This paper is organized as follows. In Section~\ref{main result}, we present the main results of this paper. Section~\ref{Preliminaries} collects the necessary preliminaries required for the proofs, including a transformation formula for $f(q^\ell)$, where $q = e^{2\pi i\left(\frac{h}{k} + \frac{iz}{k^2}\right)}$, $\,\frac{h}{k} \in \mathbb{Q}$, and $\Re(z) > 0$, as well as Weil's bound for generalized Kloosterman sums.  Section~\ref{proof of main results} is devoted to the proofs of the results stated in Section~\ref{main result}. Finally, we provide a numerical verification of our results at the end.

\section{Main Results}\label{main result}
We divide this section in two folds.   First,  we present a Rademacher-type exact formula for the $r$-colored $\ell$-regular partition function along with its applications.  At the end,   we discuss higher order Tur\'{a}n inequalities for $r$-colored $\ell$-regular partition function and their corresponding applications.
%\subsection{Rademacher-type exact formula for  $r$-colored  $\ell$-regular partitions}

The first result gives a Rademacher-type exact formula for  $r$-colored $\ell$-regular partitions.
\begin{theorem} \label{l regular r colour } Let $r \geq 1,   \ell \geq 2$ be two positive integers and $b_\ell^{(r)}(n)$ be the number of  $r$-colored $\ell$-regular partitions of $n$ and $p^{(r)}(n)$ be the number $r$-colored ordinary partitions of $n$.  Let $a^{(r)}(n)$ be the coefficient of $q^n$ in the power series expansion of $(q)_\infty^r$. An exact formula for $b_\ell^{(r)}(n)$ is given by 
{\allowdisplaybreaks
\begin{align}
b_\ell^{(r)}(n) &=\sum_{\substack{Q \mid \ell \\ Q^2<\ell}} \left(\frac{Q}{\ell}\right)^{\frac{r}{2}} \sum_{\substack{k=1\\ (k,\ell) = Q}}^\infty \sideset{}{'}\sum_{m,s=0 }^{\left\lfloor{\frac{r}{24} \left( \frac{\ell}{Q^2}-1\right)} \right\rfloor} p^{(r)}(m) a^{(r)}(s)C(n)\nonumber  \\
&\times \frac{\mathrm{d}}{\mathrm{d}n}J_0\left( \frac{4i \pi}{k}\sqrt{\left(\frac{r}{24}\left(1-\frac{Q^2}{\ell}\right)- \left( m+ \frac{Q^2s}{\ell} \right) \right) \left( n+\frac{r}{24}(\ell-1)\right)}\right),  \label{exact formula for r-colored l-regular}
\end{align}}
where
\begin{align*}
C(n) := C_{k,r,\ell,  Q,m,s}(n) = \sum_{\substack{h~\mathrm{mod}~k \\ \gcd(h,k)=1}} e^{i\pi r (s(h,k) - s(\frac{\ell h}{Q},\frac{k}{Q})) + \frac{2\pi i}{k}( mh' + s Q h_Q -nh)},
\end{align*}
with  $hh' \equiv -1(\mathrm{mod}~k)$ and $\frac{\ell}{Q} hh_Q \equiv -1(\mathrm{mod}~\frac{k}{Q})$.  Here,  the $\sum^{'}$ means the sum is running over all those $m$ and $s$,  which satisfy $ m+\frac{Q^2s}{\ell}  < \frac{r}{24}\left(1-\frac{Q^2}{\ell}\right)$. 
\end{theorem}

As an immediate implication of the above result,  we have the following asymptotic result for $b_\ell^{(r)}(n)$.

\begin{corollary}\label{asymptotic of r-colored l regular partitions theorem}
As $n \rightarrow \infty$,  we have
\begin{align}\label{asymptotic of r-colored l regular partitions}
b_\ell^{(r)}(n) \sim \frac{1}{\sqrt{2}} \left( \frac{1}{\ell} \right)^{\frac{r}{2}} \left( \frac{r}{24} \left( 1-\frac{1}{\ell} \right) \right)^{\frac{1}{4}} \left(\frac{1}{n} \right) ^{\frac{3}{4}} e^{\left( 4 \pi \sqrt{\frac{nr}{24}\left(1- \frac{1}{\ell}\right)} \right)}. 
\end{align}
\end{corollary}

Substituting $r=1$ in Theorem \ref{l regular r colour }, we recover a Rademacher-type exact formula for $\ell$-regular partitions obtained by Hagis \cite[Theorem 6]{Hagis}.  

\begin{corollary}\label{l regular partition}
Let $b_{\ell}(n)$ be the number of $\ell$-regular partitions of  $n$.  An exact formula for $b_{\ell}(n)$ is given by
\begin{align*}
b_{\ell}(n) &=\sum_{\substack{Q \mid \ell \\  Q^2<\ell}} \left(\frac{Q}{\ell}\right)^{\frac{1}{2}} \sum_{\substack{k=1\\ (k,\ell) = Q}}^\infty \sideset{}{'}\sum_{m,s=0 }^{\left\lfloor{\frac{1}{24} \left( \frac{\ell}{Q^2}-1\right)} \right\rfloor} p(m) a^{(1)}(s)D(n)\nonumber  \\
&\times \frac{\mathrm{d}}{\mathrm{d}n}J_0\left( \frac{4i \pi}{k}\sqrt{\left(\frac{1}{24}\left(1-\frac{Q^2}{\ell}\right)- \left( m+ \frac{Q^2s}{\ell} \right) \right) \left( n+\frac{1}{24}(\ell-1)\right)}\right),
\end{align*}
where
\begin{align*}
D(n) := C_{k, 1,\ell,  Q,m,s}(n) = \sum_{\substack{h~\mathrm{mod}~k \\ \gcd(h,k)=1}} e^{i\pi \left(s(h,k) - s\left(\frac{\ell h}{Q},\frac{k}{Q}\right)\right) + \frac{2\pi i}{k}( mh' + s Q h_Q -nh)},
\end{align*}
with  $hh' \equiv -1(\mathrm{mod}~k)$ and $\frac{\ell}{Q} hh_Q \equiv -1(\mathrm{mod}~\frac{k}{Q})$. Here, the $\sum^{'}$ means the sum is running over all those $m, s$ which satisfy $ m+\frac{Q^2s}{\ell}  < \frac{1}{24}\left(1-\frac{Q^2}{\ell}\right)$.  
%Also, $a(s)$ is the coefficients in the power series of $\frac{1}{f(q)}$.
\end{corollary}

Further,  letting $\ell = 2$ in  Theorem \ref{l regular r colour },  we obtain the following  Rademacher-type exact formula for $r$-colored distinct part partition function.
\begin{corollary}\label{r_colour_distint_partitions}
Let $p_d^{(r)}(n)$ be the number of $r$-colored distinct part partitions of $n,$ where parts are distinct but allowed to appear in $r$ different colors.  A Rademacher-type exact formula for $p_d^{(r)}(n)$ is given by
\begin{align*}
p_d^{(r)}(n) = \frac{1}{2^{\frac{r}{2}}} \sum_{\substack{k=1\\k ~\mathrm{odd}}}^{\infty} \sum_{\substack{m,s=0\\ 2m+s < \frac{r}{24}}}^{\lfloor{\frac{r}{24}}\rfloor} p^{(r)}(m) a^{(r)}(s) E(n) \frac{\mathrm{d}}{\mathrm{d}n}J_0\left(\frac{4\pi i}{k}\sqrt{\left(\frac{r}{48}-\frac{s}{2}-m\right)\left(\frac{r}{24} +n\right)} \right),
\end{align*}
where
\begin{align*}
E(n):= C_{k,r,2,   1,m,s}(n) = \sum_{\substack{h~\mathrm{mod}~k \\ \gcd(h,k)=1}} e^{i \pi r (s(h, k)-s(2h, k)) + \frac{2\pi i}{k}(s h_1 + mh' -nh)},
\end{align*}
with $hh' \equiv -1(\mathrm{mod}~k)$, $2hh_1 \equiv -1(\mathrm{mod}~k)$.  
% and $s(h,k)$ is as defined in \eqref{dedekind_sum}. 
%Also, $a^{(r)}(\ell)$ is the coefficients in the power series of $\frac{1}{F(q)}$.
\end{corollary}

\begin{remark}
Substituting $r=1$ in Corollary \ref{r_colour_distint_partitions},  we can easily recover the exact formula \eqref{exact_formula_pd(n)} for the distinct part partition function $p_d(n)$.
\end{remark}

\begin{remark} Letting $r=2$ in Corollary \ref{r_colour_distint_partitions},  we get back a result of Grabner and Knopfmacher  \cite[Theorem 3]{GK2006},  namely,  a Rademacher-type exact formula for a restricted partition function which is   the sum of the smallest gaps in all unrestricted partitions.  This aforementioned function was later studied by Andrews and Newmann \cite{AN19} and named as $\sigma\textrm{mex}(n)$,  sum of the minimal excludants over all the partitions of $n$. Mainly,  Grabner and Knopfmacher showed that 
\begin{align*}
\sigma\textrm{mex}(n) = \frac{\pi}{2\sqrt{6}\sqrt{n + \tfrac{1}{12}}}
\sum_{k=1}^{\infty} \frac{F_{2k-1}(n)}{2k - 1}
I_1 \left( \frac{\pi}{2k - 1} \sqrt{\frac{2}{3}\left( n + \frac{1}{12}\right)} \right) ,
\end{align*}
where $I_1$ denotes the modified Bessel function of the first kind and 
\begin{align*}
F_k(n) :=  C_{k, 2, 2, 1, 0, 0}(n) = \sum_{\substack{h~\mathrm{mod}~k \\ \gcd(h,k)=1}} e^{2\pi i  \left(s(h,k) - s(2h,k) \right) -2n \pi i\frac{h}{k}}. 
\end{align*}
\end{remark}

Letting $r=3$ in Corollary \ref{r_colour_distint_partitions}, we obtain an exact formula for $\sigma\overline{mex}(n)$, which is the sum of minimal excludants over all the overpartitions of $n$,  studied by Aricheta and Donato \cite[Definition 1]{AD2024}.  
%The exact formula for $\overline{\sigma mex}(n)$ has not been studied yet.
\begin{corollary}\label{sigma_bar_mex_Exact_formula_corollary}
A Rademacher-type exact formula for $\sigma\overline{mex}(n)$ is 
\begin{align}\label{sigma_bar_mex_Exact_formula_equation}
\sigma\overline{mex}(n) = \frac{\pi}{4 \sqrt{2}\sqrt{n + \frac{1}{8}}} \sum_{k=1}^{\infty} \frac{G_{2k-1}(n)}{2k-1}I_1 \left( \frac{\pi }{2k-1} \sqrt{n + \frac{1}{8}} \right),
\end{align}
where
\begin{align*}
G_k(n) := C_{k, 3, 2, 1, 0, 0}(n)=  \sum_{\substack{ h \bmod k \\ \gcd(h,k)=1}} e^{3\pi i \left( s(h,k) - s(2h,k) \right)-2n\pi i  \frac{h}{k}}.
\end{align*}
\end{corollary}

The final result gives higher order Tur\'{a}n inequalities for the $r$-colored $\ell$-regular partition function $b_\ell^{(r)}(n)$.  
\begin{theorem}\label{hyperbolicity of Jensen polynomial of r-colored l regular partitions}
For any positive integer $d \geq 1$, $J_{b_\ell^{(r)}}^{d,n}(X)$ is hyperbolic for all but finitely many values of n.
\end{theorem}

\section{Preliminaries}\label{Preliminaries}
In this section, we present several key results that are crucial for deriving the main results of the current paper.  We first state an important result related to Dedekind sums. 
\begin{lemma}\label{important_result_for_dedekind_sum}
If $h_1$ is an integer such that $hh_1 \equiv 1 (\mathrm{mod}~ k)$, then $s(h,k) = s(h_1,k) $.
\end{lemma}
\begin{proof}
A proof of this result can be found in \cite[Theorem 3.6(b)]{Apostal1990}. 
\end{proof}
The next result gives Weil's bound for Kloosterman's sum. This bound plays a crucial role in proving the convergence of an infinite series.
\begin{lemma}\label{Weil_bound_for_kloosterman_sum}
Let $a, b, c$ be integers with $c>0$. Then we have
\begin{align*}
 \sideset{}{'}\sum_{x\bmod c} e^{2\pi i \frac{ax+b\bar{x}}{c}} = \mathcal{O}(c^{\frac{1}{2}}d(c)\gcd(a,b,c)^{\frac{1}{2}}),  
\end{align*}
where $\bar{x}$ denotes the multiplicative inverse of $x$ modulo $c$, i.e. $x \bar{x} \equiv 1\pmod c$ and $d(c)$ denotes the number of divisors of $c$.   Here, the symbol $\sum'$ indicates that the summation is restricted to a reduced residue system modulo $c$.  
\end{lemma}
\begin{proof}
We refer \cite[Corollary 11.12]{IK2004} for the proof of this result.  
\end{proof}
Hardy and Ramanujan \cite[Lemma 4.31]{HR1918} established a transformation formula for the generating function of $p(n)$,  which was crucial in obtaining  Rademacher's exact formula for $p(n)$.  The same transformation formula will also be useful to obtain an exact formula for $r$-colored $\ell$-regular partition function.  Thus,  for clarity of the reader,  here we state the transformation formula for $f(q)=\frac{1}{(q)_\infty}$.  For $\tfrac{h}{k} \in \mathbb{Q}, \ \Re(z) > 0$,   they proved that 
\begin{align}\label{final_transformation_formula_of_f}
f\left(e^{2\pi i \left(\frac{h}{k} + \frac{iz}{k^2} \right)}\right) = \sqrt{\frac{z}{k}} e^{i\pi s(h,k)} e^{\frac{\pi}{12k}\left(\frac{k}{z}-\frac{z}{k}\right)}f\left(e^{2\pi i \left(\frac{h'}{k}+ \frac{i}{z}\right)}\right),
\end{align}
with $hh' \equiv -1(\mathrm{mod}~k)$. We now derive an analogous transformation formula for $R\left(e^{2\pi i \left(\frac{h}{k} + \frac{iz}{k^2} \right)}\right)$,  where $R(q)$ is the generating function for  $r$-colored $\ell$-regular partition function $b_{\ell}^{(r)}(n)$, given by \eqref{generating_b_l^{(r)}}, 
\begin{align}\label{gen_b_l^{(r)}}
R(q) = (L(q))^r,  
\end{align}
where
$$
L(q)=  \frac{f(q)}{f(q^{\ell})}= \frac{(q^\ell; q^\ell)_\infty}{(q; q)_\infty},
$$
which is the generating function for $\ell$-regular partition function $b_\ell(n)$.  Before deriving the transformation formula for $R(q)$,  we first establish one for $L(q)$. To this end, we require a transformation formula for $f(q^{\ell})$.  Quite interestingly,  we will see that the transformation formula for $f(q^{\ell})$ is a piecewise function consisting of  $d(\ell)$ cases,  where $d(\ell)$ denotes the number of positive divisors of $\ell$.

It is well known that the Dedekind eta function $\eta(\tau)$ has a close connection with the partition generating function.  Mainly,   for $\tau \in \mathbb{H}$,  it is defined as 
\begin{align*}
\eta (\tau) = e^{ \frac{ i \pi \tau}{12}}  \prod_{m=1}^{\infty}(1-e^{2\pi i m \tau}) = \frac{e^{ \frac{i \pi \tau}{12}}}{f(e^{2\pi i \tau})}.  
\end{align*}
For a divisor $Q$ of $\ell$,   we replace $\tau$ by $\frac{\ell}{Q} \left( \frac{h}{k}+ \frac{iz}{k}\right)$ to see that 
\begin{align}\label{relation of f and eta for blr(n)}
f\left(e^{2\pi i \frac{\ell}{Q} \left(\frac{h}{k} + \frac{iz}{k} \right)}\right) = e^{\frac{i \pi \ell }{12 Q} \left(\frac{h}{k} + \frac{iz}{k} \right) } \eta^{-1} \left( \frac{\ell}{Q}  \left(\frac{h}{k} + \frac{iz}{k} \right) \right).
\end{align}
We now recall a transformation formula for $\eta(\tau)$.  For  $\tau' \in \mathbb{H}$,  
\begin{align}\label{transformation_formula_of_eta}
\eta\left(\frac{a\tau' +b}{c \tau' +d}\right) = \epsilon(a,b,c,d) \sqrt{\frac{c\tau' +d}{i}} \eta(\tau'), \quad \forall \begin{bmatrix}
a & b \\ c & d 
\end{bmatrix} \in SL_2(\mathbb{Z}),  
\end{align}
where $\epsilon (a,b,c,d)$ is defined as 
\begin{align*}
\epsilon (a,b,c,d) = \exp \Bigg( \frac{i \pi }{12}\Phi \begin{bmatrix}
a & b \\ c & d 
\end{bmatrix} \Bigg),
\end{align*}
with
\begin{align*}
\Phi \begin{bmatrix}
a & b \\ c & d 
\end{bmatrix} = \begin{cases} 
b + 3, ~~& \mathrm{for} ~~ c=0,  d=1, \\
-b - 3, ~~& \mathrm{for} ~~ c=0,  d=-1,   \\
 \frac{a+d}{c}-12~~ \mathrm{sign}(c)  s(d,|c|),~~ & \mathrm{for} ~~ c\neq 0,
\end{cases}
\end{align*} 
and $s(h,k)$ is the Dedekind sum defined in \eqref{dedekind_sum}.  To utilize the above transformation formula for $\eta(\tau)$, we require
\begin{align*}
\frac{a\tau' +b}{c \tau' +d}=\frac{\ell}{Q} \left(\frac{h}{k} + \frac{iz}{k} \right),
\end{align*}
which suggest that,  one should consider $\tau' = \frac{h_Q}{k} + \frac{i}{\ell kz}$ for some $h_Q \in \mathbb{Z}$ and 
\begin{align}\label{values of abcd for blr(n)}
\begin{bmatrix}
a & b \\ c & d
\end{bmatrix} = \begin{bmatrix}
\frac{\ell}{Q}h & -\frac{\frac{\ell}{Q} hh_Q+1}{k} \\ k & -h_Q
\end{bmatrix}.  
\end{align}
 In order for $\frac{\frac{\ell}{Q} hh_Q+1}{k}$ to be an integer, we impose the condition that $\frac{\ell}{Q} hh_Q \equiv -1(\mathrm{mod}~k)$. Note that $c=k\neq0$ implies that $\mathrm{sign}(k) = 1$. Therefore,  we have
\begin{align}\label{epsilon (a,b,c,d) for blr(n)}
\epsilon (a,b,c,d)=\exp \left(\frac{i\pi}{12}\left(\frac{\frac{\ell}{Q} h-h_Q}{k}-12s(-h_Q,k)\right) \right).
\end{align}
Now we make use of Lemma \ref{important_result_for_dedekind_sum} to see that $s(-h_Q,k) = s\left(\frac{\ell h}{Q}, k\right)$.  Hence,  utilizing \eqref{transformation_formula_of_eta}, \eqref{values of abcd for blr(n)} and \eqref{epsilon (a,b,c,d) for blr(n)} in \eqref{relation of f and eta for blr(n)},  we get
\begin{align*}
f\left(e^{2\pi i \frac{\ell}{Q} \left(\frac{h}{k} + \frac{iz}{k} \right)}\right) = \sqrt{\frac{\ell z}{Q}} e^{i\pi s\left(\frac{\ell h}{Q},  k\right) + \frac{\pi}{12k}\left(\frac{Q}{\ell z}-\frac{\ell z}{Q}\right)}f\left(e^{2\pi i \left(\frac{h_Q}{k}+ \frac{iQ}{\ell kz}\right)}\right).
\end{align*}
Now,  replace $k$ by $\frac{k}{Q}$ and then $z$ by $\frac{z}{k}$ to have
\begin{align}\label{final_transformation_formula_of_f_for_blr(n)}
f\left(e^{2\pi i \ell \left(\frac{h}{k} + \frac{iz}{k^2} \right)}\right) = \sqrt{\frac{\ell z}{k Q}} e^{i\pi s\left(\frac{\ell h}{Q},\frac{k}{Q}\right) + \frac{\pi}{12k}\left(\frac{Q^2k}{\ell z}-\frac{\ell z}{k}\right)}f\left(e^{2\pi i Q\left(\frac{h_Q}{k}+ \frac{iQ}{\ell z}\right)}\right),  
\end{align}
where $\frac{\ell}{Q} hh_Q \equiv -1(\mathrm{mod}~\frac{k}{Q})$.   
% This is because $k$ is being replaced by $k/Q$.  
Here,  we must pay attention to the fact that for each divisor $Q$ of $\ell$ with $\gcd(k, \ell) = Q$,  we will get a different transformation formula.
At this juncture,  dividing \eqref{final_transformation_formula_of_f} by \eqref{final_transformation_formula_of_f_for_blr(n)},  we obtain
\begin{align*}
L\left(e^{2\pi i \left(\frac{h}{k} + \frac{iz}{k^2} \right)}\right) = \sqrt{\frac{Q}{\ell}} e^{i\pi \left(s(h,k) - s\left(\frac{\ell h}{Q},\frac{k}{Q}\right)\right) +  \frac{\pi}{12z}\left(1-\frac{Q^2}{\ell} \right) +\frac{\pi z}{12k^2}(\ell-1) } \frac{f\left(e^{2\pi i \left(\frac{h'}{k}+ \frac{i}{ z}\right)}\right)}{f\left(e^{2\pi i Q\left(\frac{h_Q}{k}+ \frac{iQ}{\ell z}\right)}\right)}.
\end{align*}
Finally,  raising $r$-th power on both sides, we derive the following transformation formula for $R(q)$.   
\begin{lemma} \label{trans_for_R(q)}
Let $r,  k$ and $\ell \geq 2$ be positive integers and $\gcd(\ell,  k) = Q$.  Let $h_{Q}$ be an integer such that $\frac{\ell}{Q} hh_Q \equiv -1(\mathrm{mod}~\frac{k}{Q})$,  then we have the following transformation formula for $R(q)$: 
\begin{align*}
R\left(e^{2\pi i \left(\frac{h}{k} + \frac{iz}{k^2} \right)}\right) = \left(\frac{Q}{\ell}\right)^{\frac{r}{2}} e^{i\pi r \left(s(h,k) - s\left(\frac{\ell h}{Q},\frac{k}{Q}\right)\right)+\frac{\pi r}{12z}\left(1-\frac{Q^2}{\ell} \right) +\frac{\pi r z}{12k^2}(\ell-1) } \frac{F\left(e^{2\pi i \left(\frac{h'}{k}+ \frac{i}{ z}\right)}\right)}{F\left(e^{2\pi i Q\left(\frac{h_Q}{k}+ \frac{iQ}{\ell z}\right)}\right)},
\end{align*}  
where $hh' \equiv -1(\mathrm{mod}~k)$ and $F(q)=f(q)^r= (q)_\infty^{-r}$.
\end{lemma}

Now we state a result, from the theory of Bessel functions,  which will be useful to prove our main result.  
\begin{lemma} \cite[p. 181, Equation (1)]{GN1995}
For $c>0$ and $\Re(\nu) >0,$ an integral representation of the modified Bessel function of the first kind is given by:
 %{\bf Conditions on $\nu$ and $z$}
\begin{align}\label{bessel_I_integral}
I_{\nu}(z) = \frac{(z/2)^{\nu}}{2\pi i} \int_{c-\infty}^{c+i\infty} t^{-\nu -1} e^{t+\frac{z^2}{4t}} \mathrm{d}t.
\end{align}
\end{lemma}
Prior to stating the next result, we introduce several notations that will be used in the subsequent discussion.  First,  we define 
\begin{align*}
\omega(h,k) = e^{\pi i  s(h,k)}.
\end{align*}
We then set
\begin{align}\label{Defn_omega_h_k_l_r}
\omega(h, k, \ell, r)= \left[\frac{\omega(h,k)}{\omega\left(\frac{\ell h}{Q}, \frac{k}{Q}\right)}\right]^r, 
\end{align}
where $Q=\gcd(k,\ell). $ Now one can write $k=QK$ and $\ell=QT$ with $\gcd(K,T)=1$. Then the above expression becomes
\begin{align}\label{omega in term of T and K}
\omega(h, k,  \ell,  r)= \left[\frac{\omega(h,k)}{\omega\left(Th, K\right)}\right]^r.
\end{align}
Our objective is  to estimate a sum involving $\omega(h, k,  \ell,  r)$. To this end, we begin with the following result associated to $\omega(h, k)$, the proof of which can be found in \cite{H70}.
\begin{proposition} \label{proposition 1}
If $k$ is odd, then
\begin{align} \label{omega for odd k}
\omega(h, k) = \left( \frac{h}{k} \right) i^{(k - 1)/2} e^{ \frac{2\pi i p(h - h')}{gk}}.
\end{align}
If $k$ is even, then
\begin{align}\label{omega for even k}
\omega(h, k) = \left( \frac{k}{h} \right) i^{b(k + 1)/2} e^{ \frac{2\pi i p(h - h')}{gk} }.
\end{align}
Here
\begin{align*}
g = \begin{cases}
\gcd(3, k) & \textrm{if}~k~\textrm{odd},  \\
8 \gcd(3,k) & \textrm{if}~k~\textrm{even}, 
\end{cases}
\end{align*}
and
$h'$ satisfies $hh' \equiv -1 \pmod{gk}$, and $p$ satisfies the relation $fp \equiv 1 \pmod{gk}$ where $f = \frac{24}{g}$. In \eqref{omega for even k}, $b \equiv h' \pmod{8}$.   The symbol $ \left( \frac{h}{k} \right)$ denotes the Jacobi symbol.
\end{proposition}
Our goal is now to derive an analogue of Proposition \ref{proposition 1} for $\omega(h, k,  \ell,   r)$. To accomplish this,   we utilize basic properties of the Jacobi symbol and observe that the Proposition \ref{proposition 1} remains valid when parameters $h, h', k, g, f, p, b$ are replaced by $Th, h^*, K, G, F, P, B,$ respectively. The proof is divided into three cases.

\textbf{Case 1:} Suppose $k$ is odd, then $K$ is necessarily odd as well. From \eqref{omega in term of T and K} and \eqref{omega for odd k}, we obtain
\begin{align}\label{w-case1}
\omega(h, k,  \ell,  r) = \left[\frac{\left( \frac{h}{k} \right) i^{(k - 1)/2} e^{ \frac{2\pi i p(h - h')}{gk}}}{ \left( \frac{Th}{K} \right) i^{(K - 1)/2} e^{ \frac{2\pi i P(Th - h^*)}{GK} }}\right]^r.
\end{align}
Here 
\begin{align*}
&g=\gcd(3,k),&&~hh' \equiv -1 \pmod {gk},~&&fp \equiv 1 \pmod {gk},~&& f=24/g,  \\
&G=\gcd(3,K),&&~Thh^* \equiv -1 \pmod {GK},~&& FP \equiv 1 \pmod {GK},~&& F=24/G.
\end{align*}
Choosing $T'$ such that $TT' \equiv 1 \pmod{GK}$, we deduce that $h^* \equiv T'h' \pmod{GK}$.  Note that 
\begin{align*}
 &\quad Thh^* \equiv -1 \pmod {GK}\\
& \Rightarrow T'Thh^* \equiv -T' \pmod {GK}\\
& \Rightarrow hh^* \equiv -T' \pmod {GK} \hspace{2cm} ( \textrm{since}~~ TT' \equiv 1 \pmod{GK})\\
& \Rightarrow h'hh^* \equiv -T'h' \pmod {GK}\\
& \Rightarrow h^* \equiv T'h' \pmod{GK} \hspace{2cm} (\textrm{since}~~hh' \equiv -1 \pmod{gk}~ \mathrm{and}~ GK \mid gk).
\end{align*} 
If $g = JG$ with $J = 1$ or $3$, then $F = \frac{24}{G}= \frac{24J}{g}= Jf$.  One can see that  $P \equiv Ap \pmod{GK}$,  where $JA \equiv 1 \pmod{GK}$.  Moreover,  using properties of Jacobi symbol,  we have
\begin{align*}
\frac{\left( \frac{h}{k} \right)}{ \left( \frac{Th}{K} \right)} = \frac{\left( \frac{h}{QK} \right)}{ \left( \frac{Th}{K} \right)} = \frac{\left( \frac{h}{Q} \right) \left( \frac{h}{K} \right)}{\left( \frac{T}{K} \right) \left( \frac{h}{K} \right)} = \frac{\left( \frac{h}{Q} \right)}{ \left( \frac{T}{K} \right)}.
\end{align*}
Hence,  substituting the above expression in \eqref{w-case1} and upon simplification,  we obtain that
\begin{align*}
\omega(h, k,  \ell,   r) = \left[\frac{\left( \frac{h}{Q} \right)}{ \left( \frac{T}{K} \right)}\right]^r i^{\frac{r(k - K)}{2}} e^{\frac{2\pi i p r(Uh + Vh')}{gk} }, 
\end{align*}
where
\begin{align}\label{define U and V}
U = 1 - JA\ell, \quad V = JAT'Q - 1.
\end{align}
The term $\left[ \left( \frac{h}{Q} \right)/ \left( \frac{T}{K} \right) \right]^r$ has absolute value one and depends only on $k$ and $\ell$ when $h \equiv a \pmod{Q}$ with $\gcd(a, Q) = 1$.\\
\textbf{Case 2:} Suppose that both $k$ and $K$ are even. Then, from \eqref{omega in term of T and K} and \eqref{omega for even k}, we have
\begin{align*}
\omega(h, k,  \ell,  r) = \left[\frac{\left( \frac{k}{h} \right) i^{b(k + 1)/2} e^{ \frac{2\pi i p(h - h')}{gk} }}{\left( \frac{K}{Th} \right) i^{B(K + 1)/2} e^{ \frac{2\pi i P(Th - h^*)}{GK} }}\right]^r.
\end{align*}
The parameters satisfy
\begin{align*}
&g= 8 \gcd(3,k),&& hh' \equiv -1 (\bmod~ {gk}),&& fp \equiv 1 (\bmod~{gk}),&& f=\frac{24}{g},&& b \equiv h' (\bmod~{8}),\\
&G=8 \gcd(3,K),&& Thh^* \equiv -1 (\bmod~{GK}),&& FP \equiv 1 (\bmod~{GK}),&& F=\frac{24}{G},&& B \equiv h^* (\bmod ~{8}).
\end{align*}
Let us consider $T'$ such that $TT' \equiv 1 \pmod{GK}$, then one can check that  $h^* \equiv T'h' \pmod{GK}$ and hence $B \equiv h^* \equiv T'b \pmod{8}$ since $8 \mid G$. If $g = JG$ ($J = 1$ or $3$) then $F = Jf$ and $P \equiv Ap \pmod{GK}$ where $A$ is defined as in Case 1. Writing $Q = 2^\alpha Q^*$ with $\alpha \geq 0$ and $Q^*$ odd, we observe that 
 \begin{align*}
\frac{\left( \frac{k}{h} \right)}{ \left( \frac{K}{Th} \right)} = \frac{\left( \frac{QK}{h} \right)}{ \left( \frac{K}{Th} \right)} = \frac{\left( \frac{Q}{h} \right) \left( \frac{K}{h} \right)}{\left( \frac{K}{T} \right) \left( \frac{K}{h} \right)} = \frac{\left( \frac{Q}{h} \right)}{ \left( \frac{K}{T} \right)}=\frac{\left( \frac{2^\alpha}{h} \right) \left( \frac{Q^*}{h} \right)}{ \left( \frac{K}{T} \right)}=\frac{\left( \frac{2^\alpha}{h} \right) \left( \frac{h}{Q^*} \right)}{ \left( \frac{K}{T} \right)}(-1)^{\frac{(h-1)(Q^*-1)}{4}}, 
\end{align*}
where
\begin{align*}
\left(\frac{2^\alpha}{h}\right) = 
\begin{cases}
1, & \text{if } \alpha \text{ is even},\\[3pt]
(-1)^{(h^2 - 1)/8}, & \text{if } \alpha \text{ is odd.}
\end{cases}
\end{align*}
Note that $h$ is odd since $hh'\equiv -1 (\bmod~ 8)$. 
Therefore,  we finally have
\begin{align*}
\omega(h, k,  \ell,  r) = \left[\frac{\left( \frac{2^\alpha}{h} \right) \left( \frac{h}{Q^*} \right)}{ \left( \frac{K}{T} \right)} \right]^r(-1)^{ \frac{r(h-1)(Q^* - 1)}{4}} i^{ \frac{rb}{2} \{ (k+1) -T'(K+1) \} } e^{ \frac{ 2\pi i p r (Uh + Vh')}{gk} },
\end{align*}
with $U$ and $V$ are defined as in \eqref{define U and V}.
Note that the quotients of  Jacobi symbols  in the above expression has absolute value one and depends only on $k$ and $\ell$ under the constraints $h \equiv a \pmod Q$ where $\gcd(a, Q) = 1$, and $h \equiv d  \pmod 8$ with $d$ odd.\\
\textbf{Case 3:} The final case remains when $k$ is even and $K$ is odd.  Proceeding along the same line as we did in {\bf Case 1} and {\bf Case 2},  one can show that
\begin{align*}
\omega(h, k, \ell, r) = \left[\frac{\left( \frac{k}{h} \right) i^{b(k + 1)/2} e^{ \frac{2\pi i p(h - h')}{gk} }}{\left( \frac{Th}{K} \right) i^{(K - 1)/2} e^{\frac{2\pi i P(Th - h^*)}{GK} }}\right]^r,
\end{align*}
where
\begin{align*}
&g=8 \gcd(3,k),&&7 hh' \equiv -1 (\bmod~{gk}),&& fp \equiv 1 (\bmod~{gk}),&& f=24/g, b \equiv h' (\bmod~{8})\\
&G= \gcd(3,K),&& Thh^* \equiv -1 (\bmod~{GK}),&& FP \equiv 1 (\bmod {GK}),&& F=24/G.
\end{align*}
Letting $T'$ such that $TT' \equiv 1 \pmod{GK}$,  we can show that  $h^* \equiv T'h' \pmod{GK}$.  Write $g = JG$ ($J = 8$ or $24$),  then $F = Jf$ and $P \equiv Ap \pmod{GK}$ where $JA \equiv 1 \pmod{GK}$. Writing $Q = 2^\alpha Q^*$,  with $Q^*$ odd,  we see that
\begin{align*}
\frac{\left( \frac{k}{h} \right)}{ \left( \frac{Th}{K} \right)} = \frac{\left( \frac{QK}{h} \right)}{ \left( \frac{Th}{K} \right)} = \frac{\left( \frac{Q}{h} \right) \left( \frac{K}{h} \right)}{\left( \frac{T}{K} \right) \left( \frac{h}{K} \right)} = \frac{\left( \frac{Q}{h} \right)}{ \left( \frac{T}{K} \right)}(-1)^{(h-1)(K-1)/4} & =\frac{\left( \frac{2^\alpha}{h} \right) \left( \frac{Q^*}{h} \right)}{ \left( \frac{T}{K} \right)}(-1)^{(h-1)(K-1)/4} \nonumber \\
&=\frac{\left( \frac{2^\alpha}{h} \right) \left( \frac{h}{Q^*} \right)}{ \left( \frac{T}{K} \right)}(-1)^{\frac{(h-1)(K-Q^*)}{4}}.
\end{align*}
Using the above expression,  we conclude that  
\begin{align*}
\omega(h, k,  \ell,  r) = \left[\frac{\left( \frac{2^\alpha}{h} \right) \left( \frac{h}{Q^*} \right)}{ \left( \frac{T}{K} \right)}\right]^r (-1)^{ \frac{r(h-1)(K - Q^*)}{4}} i^{br(k+1)/2-r(K-1)/2}e^{ \frac{2\pi i p r (Uh + Vh')}{gk} },
\end{align*}
where $U$ and $V$ are defined as in \eqref{define U and V}.  
The expression $\left[\frac{\left( \frac{2^\alpha}{h} \right) \left( \frac{h}{Q^*} \right)}{ \left( \frac{T}{K} \right)}\right]^r$
has absolute value one and depends only on $k,  \ell$ and if $h \equiv a \pmod{Q}$ where $\gcd(a, Q) = 1$, and $h \equiv d \pmod{8}$ where $d$ is odd. \\
We now summarize the above three cases in the following proposition. 
\begin{proposition}\label{proposition 2}
Let $\omega(h, k, \ell, r)$ be the function defined in \eqref{Defn_omega_h_k_l_r}.  Considering the variables defined above,   we have 
$$\omega(h, k, \ell, r) = A(h, k, \ell,  r) e^{2\pi i q r \left( \frac{Uh + Vh'}{gk} \right) },$$ 
where $A(h, k, \ell,  r)$ is some complex number with $|A(h, k, \ell, r)| = 1$ and it  depends only on $k$ and $\ell$ if $h \equiv a \pmod{Q}$ where $\gcd(a, Q) = 1$ and also, if $k$ is even, $h \equiv d \pmod{8}$ where $d$ is odd. If $g = JG$ ($J = 1, 3, 8, 24$) and $JA \equiv 1 \pmod{GK}$, then $U = 1 - JA\ell$ and $V = JAT'Q - 1$.
\end{proposition}
We are now prepared to state a bound for a Kloosterman-type sum that will be employed in the subsequent analysis.
\begin{theorem}\label{bounds for generalized kloosterman sum}
Let $\gcd(k, \ell) = Q$, and suppose $h \equiv a \pmod{Q}$, $\gcd(a, Q) = 1$. Assume further that $hh' \equiv -1 \pmod{k}$ and that $s_1 \leq h' \pmod{k} < s_2 $ for integers satisfying $0 \leq s_1 < s_2 \leq k$. Let $t \leq n$; and $M$ is a fixed integer, then the exponential sum
\begin{align*}
X:= W_{\ell,  r,  M}(n,  k)   = \sideset{}{'}\sum_{h \bmod{k}} \omega(h, k, \ell,  r) e^{-\frac{2\pi i (hn - h'M)}{k}}
\end{align*}
admits the estimate
\begin{align*}
X:=W_{\ell,  r,  M}(n,  k)   = \mathcal{O}\left(n^{1/2}k^{1/2 + \epsilon} \right), 
\end{align*}
where the implied constant in the $\mathcal{O}$-term depends only on $\ell$. Here, the symbol $\sum'$ indicates that the summation is restricted to a reduced residue system modulo $k$, possibly subject to additional stated conditions.
\end{theorem}
\begin{proof} The proof of this result follows the same line of reasoning as that given by Hagis \cite[p.~378]{Hagis}. However, for the sake of completeness and clarity, we provide a brief outline below.  Using a property of the Dedekind sum,  one can see that  the function $\omega(h, k, \ell,  r)$ has period $k$ when viewed as a function of $h$. If we change the sum over $\bmod~k$ in $X$ to $ \bmod~gk$ and select $h'$ so that $hh' \equiv -1 (\bmod~gk)$, then employing Proposition \ref{proposition 2},  we obtain
 $$
 X= g^{-1} \sideset{}{'}\sum_{h \bmod gk} A(h, k, \ell,  r)e^{ \frac{2\pi i f(h)}{gk} },
 $$
  where $f(h)=(qrU-gn)h+(qrV+gM)h'.$
If $k$ is even,  we split $X$ into four parts, $X_1$, $X_3$, $X_5$, $X_7$, so that in $X_d$ we have $h \equiv d \pmod{8}$ as well as $h \equiv a \pmod{Q}$. Therefore,  $X = X_1 + X_3 + X_5 + X_7$,  where 
\begin{align}\label{define X_d}
X_d = A_d    \sideset{}{'}\sum_{h \bmod gk}  e^{\frac{2\pi i f(h)}{gk} }
\end{align}
with $|A_d| = g^{-1} \leq 1$.
If $k$ is odd,  then \eqref{define X_d} holds if we identify $X$ with $X_d$ and ignore the restriction $h \equiv d \pmod{8}$.
For two fixed integers $s_1$ and $s_2$,  let us define a function $\beta(s)$ as follows: 
\begin{align*}
\beta(s) =\begin{cases}
1, & \textrm{if}~ s_1 \leq s~(\bmod~{k}) < s_2,  \\
0, & \textrm{otherwise}.  
\end{cases}
\end{align*}
Note that $\beta(s)$ is a periodic function with period $k$. Thus,  $\beta(s)$ can be written as finite Fourier series,   
$$
\beta(s) = \sum_{j = 0}^{k - 1} \beta_j e^{\frac{2\pi i sj}{k} },
\quad \textrm{where} \quad
 \beta_j = \frac{1}{k} \sum_{s = 0}^{k - 1} \beta(s) e^{\frac{-2\pi i sj}{k} }.
 $$ 
One can prove that 
%$\sum_{j=0}^{k-1} \beta = \mathcal{O}(\log k)$, 
 $\sum_{j=0}^{k-1} \beta_j = \mathcal{O}(k^\varepsilon)$ for any $\varepsilon > 0$.
We can now drop the restriction $s_1 \leq h' \pmod{k} < s_2 $ and write
\begin{align*}
X_d &= A_d \sideset{}{'}\sum_{h \bmod gk} \beta(h') e^{\frac{2\pi i f(h)}{gk} }\\
&= A_d \sum_{j=0}^{k-1} \beta_j  \sideset{}{'}\sum_{h \bmod gk} e^{\frac{2\pi i \left((qrU-gn)h+(qrV+gM+gj)h'\right)}{gk}}.
\end{align*}
If $k$ is odd, then $\sum'$ represents a Kloosterman sum. 
When $k$ is even, we write $Q = 2^{\alpha} Q^{*}$, where $\alpha \geq 0$ and $Q^{*}$ is odd. It is straightforward to verify that if $\alpha \leq 1$, the two congruence conditions
\begin{center}
(I) $h \equiv a \pmod{Q} \quad$ and $\quad$ (II) $h \equiv d \pmod{8}$ 
\end{center} 
are jointly equivalent to a single condition of the form 
$$
h \equiv a^{*} \pmod{8Q^{*}}.
$$

If $\alpha = 2$ and $a \not\equiv d \pmod{4}$, then the sum $\sum'$ is empty. However, if $a \equiv d \pmod{4}$, conditions (I) and (II) again combine into a single congruence
$$
h \equiv a^{*} \pmod{8Q^{*}}.
$$

For $\alpha \geq 3$, if $a \not\equiv d \pmod{8}$, the sum $\sum'$ is empty; 
whereas if $a \equiv d \pmod{8}$, conditions (I) and (II) reduce to 
$$
h \equiv a \pmod{Q}.
$$
Hence, in all cases, either $\sum'$ is empty or it is a Kloosterman sum. 
Finally,  applying Weil’s bound  for Kloosterman sums i.e.,  Lemma \ref{Weil_bound_for_kloosterman_sum}, we  obtain
\begin{align*}
X=\mathcal{O}(n^{1/2} k^{1/2+\epsilon}).
\end{align*}
\end{proof}

We are now ready to prove our main results.

\section{proof of main results}\label{proof of main results}

\begin{proof}[Theorem \ref{l regular r colour }][]
 Let $b_\ell^{(r)}(n)$ be the number of  $r$-colored $\ell$-regular partitions of $n$.  From the generating function  \eqref{gen_b_l^{(r)}} for $b_\ell^{(r)}(n)$ and use of Cauchy's integral formula gives 
\begin{align*}
b_\ell^{(r)}(n) = \frac{1}{2\pi i} \int_{C} \frac{R(q)} {q^{n+1}}  \mathrm{d}q,
\end{align*}
where $R(q)= \frac{F(q)}{F(q^\ell)}$ and $F(q)=  f(q)^r$ and $C:|q| = r' < 1$. Substitute $q = e^{2\pi i \tau}$ with $\Im(\tau) >0$.
Then,  we have 
\begin{align}\label{start_integral_regular_partition}
b_\ell^{(r)}(n) = \int_{\tau_0}^{\tau_0 +1} R(e^{2\pi i \tau}) e^{-2\pi i n\tau} \mathrm{d}\tau,
\end{align}
where $\tau_0 \in \mathbb{H}$ and any path from $\tau_0$ to $\tau_0 +1$ is allowed.   
 The path we choose is same as Rademacher's  path based on Farey dissection,  interested reader can see \cite{HRad1943} for more details.  
 
 We construct Ford circles belonging to Farey sequence of order $N$.  For the Ford circle $C(h,k)$,  we will choose arc $\gamma_{h,k}$ in such a way that it connects the tangency points $\alpha_1$ and $\alpha_2$ and it does not touches the real axis.  Here we denote them as  $\frac{h}{k} + C'_{h,k}$ and $\frac{h}{k}+C''_{h,k}$,  where
 \begin{align*}
C'_{h,k}:=  - \frac{k_1}{k(k^2 + k_1^2)}+ i  \frac{1}{k^2 + k_1^2}, \quad \textrm{and} \quad C''_{h,k}:= \frac{k_2}{k(k^2 + k_2^2)} + i \frac{1}{k^2 + k_2^2}.
\end{align*} 

Thus,  considering the path of integration as the union of arcs $\gamma_{h,k}$,   the integral  \eqref{start_integral_regular_partition} becomes
\begin{align*}
b_\ell^{(r)}(n) = \sum_{\substack{0\leq h<k \leq N \\\gcd(h,k)=1}} \int_{\gamma_{h,k}} R(e^{2\pi i \tau}) e^{-2\pi i n\tau} \mathrm{d}\tau,
\end{align*}
where $\frac{h}{k}$ is a fraction from Farey sequence of order $N$.  Now,  we substitute $\tau = \frac{h}{k}+\zeta$ to see that 
\begin{align*}
 b_\ell^{(r)}(n) = \sum_{\substack{0\leq h<k \leq N \\\gcd(h,k)=1}} \int_{\zeta'_{h,k}}^{\zeta''_{h,k}} R\left(e^{2\pi i \left(\frac{h}{k}+\zeta \right)}\right) e^{-2\pi i n\left(\frac{h}{k}+\zeta\right)} \mathrm{d}\zeta.
\end{align*}
We further substitute $\zeta = \frac{iz}{k^2}$.  Under these substitutions,  the circle $C(h,k) : \left|\tau - \left(\frac{h}{k}+\frac{1}{2k^2}\right) \right| = \frac{1}{2k^2}$ becomes $\left|z-\frac{1}{2}\right| = \frac{1}{2}$.  Since $z = \frac{k^2}{i}\zeta$ and $z$ varies from $z'_{h,k}$ to $z''_{h,k}$ with
\begin{align*}
z'_{h,k} =  \frac{k^2}{k^2 + k_1^2} + \frac{ikk_1}{k^2 + k_1^2},
\end{align*}
and
\begin{align*}
z''_{h,k} =  \frac{k^2}{k^2 + k_2^2} - \frac{ikk_2}{k^2 + k_2^2}.
\end{align*}
Thus,  the integral becomes
\begin{align*}
b_\ell^{(r)}(n) = \sum_{\substack{0\leq h<k \leq N \\\gcd(h,k)=1}} \frac{i}{k^2} e^{-2\pi i n \frac{h}{k}} \int_{z'_{h,k}}^{z''_{h,k}} R\left(e^{2\pi i \left(\frac{h}{k} + \frac{iz}{k^2} \right)}\right) e^{2 \pi n \frac{z}{k^2}} \mathrm{d}z.
\end{align*}
  Now, we employ the transformation formula for $ R\left(e^{2\pi i \left(\frac{h}{k} + \frac{iz}{k^2} \right)}\right)$ i.e.  Lemma \ref{trans_for_R(q)} in the above integral to see that 
 \begin{align*}
b_\ell^{(r)}(n) &=\sum_{Q \mid \ell} \left(\frac{Q}{\ell}\right)^{\frac{r}{2}} \sum_{\substack{0\leq h<k \leq N \\\gcd(h,k)=1 \\ \gcd(k, \ell) =Q}} \frac{i}{k^2} e^{i\pi r \left(s(h,k) - s\left(\frac{\ell h}{Q},\frac{k}{Q}\right)\right) -2 n \pi i \frac{ h}{k}}  \nonumber \\
& \times \int_{z'_{h,k}}^{z''_{h,k}}  e^{\left(\frac{\pi r}{12z}\left(1-\frac{Q^2}{\ell} \right) +\frac{\pi r z}{12k^2}(\ell-1) +2 \pi n \frac{z}{k^2} \right)}
 \frac{F\left(e^{2\pi i \left(\frac{h'}{k}+ \frac{i}{ z}\right)}\right)}{F\left(e^{2\pi i Q\left(\frac{h_Q}{k}+ \frac{iQ}{\ell z}\right)}\right)}  \mathrm{d}z.
 \end{align*}
Next,  we make $\Re(z)$ sufficiently small so that $\Re\left(\frac{1}{z} \right)$ becomes large. 
This necessitates splitting the series into two parts depending whether $\Bigg| e^{\frac{\pi r}{12z}\left(1-\frac{Q^2}{\ell} \right)} \Bigg| \rightarrow 0$ or not. 
This behaviour is determined by whether $1-\frac{Q^2}{\ell}$ is negative or not. Thus, we write 
\begin{align}\label{divide blr(n) in two sum}
b_\ell^{(r)}(n) = T_1 + T_2,
\end{align}
where $T_1$ is the sum with $Q^2 < \ell $ and $T_2$ is with $ Q^2 \geq \ell $.   First,  we shall evaluate $T_2$.  We know $F(q)= (q)_\infty^{-r}$.  Hence,  for $\Re(z) \rightarrow 0$,  one can see that 
\begin{align*}
F\left(e^{2\pi i \left(\frac{h'}{k}+ \frac{i}{z}\right)}\right) = \sum_{m=0}^{\infty} p^{(r)}(m) e^{2m\pi i \left(\frac{h'}{k}+\frac{i}{z} \right)} = 1+ \mathcal{O}\left( e^{-2\pi \Re(1/z)} \right) = \mathcal{O}(1).  
\end{align*}
Similarly,  one can check that 
$$
\frac{1}{F\left(e^{2\pi i Q \left(\frac{h_Q}{k}+ \frac{iQ}{\ell z}\right)}\right)} = \mathcal{O}(1).
$$
From Theorem \ref{bounds for generalized kloosterman sum},  we have
\begin{align*}
\sum_{\substack{0\leq h<k \leq N \\ \gcd(h,k)=1 \\ \gcd(k, \ell)=Q}}  e^{i\pi r \left(s(h,k) - s\left(\frac{\ell h}{Q},\frac{k}{Q}\right)\right)-2n\pi i \frac{h}{k}} = \mathcal{O}\left(n^{\frac{1}{2}} k^{\frac{1}{2}+\epsilon} \right).
\end{align*}
Utilizing the above bounds,  we get
{\allowdisplaybreaks
\begin{align*}
T_2 &= \sum_{\substack{Q \mid \ell \\  Q^2 \geq \ell }} \left(\frac{Q}{\ell}\right)^{\frac{r}{2}} \sum_{\substack{0\leq h<k \leq N \\ \gcd(h,k)=1 \\ \gcd(k, \ell)=Q}} \frac{i}{k^2} e^{i\pi r \left(s(h,k) - s\left(\frac{\ell h}{Q},\frac{k}{Q}\right)\right)} e^{-2n\pi i \frac{h}{k}} \\ 
& \times\int_{z'_{h,k}}^{z''_{h,k}}  e^{ \frac{\pi r}{12z}\left(1-\frac{Q^2}{\ell} \right) + \frac{\pi r z}{12k^2}(\ell-1) +2 \pi n \frac{z}{k^2} } 
\frac{ F\left( e^{2\pi i \left(\frac{h'}{k}+ \frac{i}{ z}\right)}\right)}{F\left(e^{2\pi i Q\left(\frac{h_Q}{k}+ \frac{iQ}{\ell z}\right)}\right) }  \mathrm{d}z. \\
&= \mathcal{O} \left( \sum_{k=1}^{N}  \frac{\left(n^{\frac{1}{2}} k^{\frac{1}{2}+\epsilon} \right)}{k^2}  \int_{z'_{h,k}}^{z''_{h,k}}  e^{  \frac{\pi r}{12k^2}(\ell-1) \mathrm{Re}(z) + \frac{2n \pi }{k^2} \mathrm{Re}(z) } \mathrm{d}z \right),
\end{align*}}
since $\left| e^{-\frac{\pi r}{12}\left(\frac{Q^2}{\ell} -1 \right) \mathrm{Re}(\frac{1}{z}) }\right| \leq 1.$
The above integrand is regular,  so we can choose any path joining $z'_{h,k}$ and $z''_{h,k}$.  We consider the path as the chord from $z'_{h,k}$ to $z''_{h,k}$ and on this chord, we have $0 < \Re(z) \leq \frac{2k^2}{N^2}$, and length of the path of integration is less than $\frac{2\sqrt{2}k}{N}$.  Thus,  using these two bounds, we get
\begin{align} \label{final_evaluation_T2}
T_2 &= \mathcal{O} \left( \sum_{k=1}^{N}  \frac{1}{k^{\frac{3}{2}-\epsilon}} e^{\frac{4\pi }{N^2}  \left( n +\frac{r}{24}(\ell-1)  \right)} \frac{2\sqrt{2}k}{N} \right) \nonumber \\
&=\mathcal{O} \left(\frac{1}{N} e^{\frac{4\pi }{N^2}  \left( n +\frac{r}{24}(\ell-1)  \right)} \sum_{k=1}^{N} \frac{1}{k^{\frac{1}{2}-\epsilon}} \right).   
\end{align}
Now, our next aim is to evaluate $T_1$, the series corresponding to $ Q^2 < \ell $.  Note that $T_1$ will give the main term of our result.  We write 
{\allowdisplaybreaks
\begin{align*}
 F\left(e^{2\pi i \left(\frac{h'}{k}+ \frac{i}{z}\right)}\right) &= \sum_{m=0}^{\left\lfloor{\frac{r}{24} \left( \frac{\ell}{Q^2}-1\right)} \right\rfloor} p^{(r)}(m) e^{2\pi im\left(\frac{h'}{k}+\frac{i}{z}\right)} + \sum_{m= \left\lfloor{\frac{r}{24} \left( \frac{\ell}{Q^2}-1\right)} \right\rfloor +1}^{\infty}  p^{(r)}(m) e^{2\pi im\left(\frac{h'}{k}+\frac{i}{z}\right)} \\
&= \sum_{m=0}^{\left\lfloor{\frac{r}{24} \left( \frac{\ell}{Q^2}-1\right)} \right\rfloor} p^{(r)}(m) e^{2\pi im\left(\frac{h'}{k}+\frac{i}{z}\right)} + \mathcal{O}\left(e^{-2\pi \left( \left\lfloor{\frac{r}{24} \left( \frac{\ell}{Q^2}-1\right)} \right\rfloor +1 \right) \Re\left(\frac{1}{z} \right)} \right).
\end{align*}}
Further,  we can write
\begin{align*}
\frac{1}{F\left(e^{2\pi i Q \left(\frac{h_Q}{k}+ \frac{iQ}{\ell z}\right)}\right)} &= \sum_{s=0}^{\infty} a_r(s)e^{2\pi isQ\left(\frac{h_Q}{k}+\frac{i Q}{\ell z}\right)} \\
&= \sum_{s=0}^{\left\lfloor{\frac{r}{24} \left( \frac{\ell}{Q^2}-1\right)} \right\rfloor} a_r(s) e^{2\pi i s Q\left(\frac{h_Q}{k}+\frac{iQ}{\ell z}\right)} + \sum_{s=\left\lfloor{\frac{r}{24} \left( \frac{\ell}{Q^2}-1\right)} \right\rfloor +1}^{\infty} a_r(s)e^{2\pi i sQ\left(\frac{h_Q}{k}+\frac{iQ}{\ell z}\right)}\\
&= \sum_{s=0}^{\left\lfloor{\frac{r}{24} \left( \frac{\ell}{Q^2}-1\right)} \right\rfloor} a_r(s) e^{2\pi i s Q\left(\frac{h_Q}{k}+\frac{iQ}{\ell z}\right)} + \mathcal{O}\left(e^{\frac{-2 \pi Q^2}{\ell} \left(\left\lfloor{\frac{r}{24} \left( \frac{\ell}{Q^2}-1\right)} \right\rfloor +1 \right) \Re\left(\frac{1}{z} \right)} \right).
\end{align*}
Thus,  combining above two bounds,  we get
{\allowdisplaybreaks
\begin{align}
\frac{F\left(e^{2\pi i \left(\frac{h'}{k}+ \frac{i}{z}\right)}\right)}{F\left(e^{2\pi i Q \left(\frac{h_Q}{k}+ \frac{iQ}{\ell z}\right)}\right)} &= \sum_{m,s=0}^{\left\lfloor{\frac{r}{24} \left( \frac{\ell}{Q^2}-1\right)} \right\rfloor} p^{(r)}(m) a_r(s)e^{\frac{2\pi i}{k}( mh' + s Q h_Q)}e^{\frac{-2\pi}{z}\left(m+\frac{Q^2}{\ell} s \right)} \nonumber \\ 
& + \mathcal{O}\left(e^{\frac{-2 \pi Q^2}{\ell} \left(\left\lfloor{\frac{r}{24} \left( \frac{\ell}{Q^2}-1\right)} \right\rfloor +1 \right) \Re\left(\frac{1}{z} \right)} \right). \label{bound for F in T1}
\end{align}}
Note that 
{\allowdisplaybreaks
\begin{align*}
T_1 &= \sum_{\substack{Q \mid \ell \\  Q^2 < \ell }} \left(\frac{Q}{\ell}\right)^{\frac{r}{2}} \sum_{\substack{0\leq h<k \leq N \\ \gcd(h,k)=1 \\ \gcd(k, \ell)=Q}} \frac{i}{k^2} e^{i\pi r \left(s(h,k) - s\left(\frac{\ell h}{Q},\frac{k}{Q}\right)\right)} e^{-2n\pi i \frac{h}{k}} \\ 
& \times\int_{z'_{h,k}}^{z''_{h,k}}  e^{ \frac{\pi r}{12z}\left(1-\frac{Q^2}{\ell} \right) + \frac{\pi r z}{12k^2}(\ell-1) +2 \pi n \frac{z}{k^2} } 
\frac{ F\left( e^{2\pi i \left(\frac{h'}{k}+ \frac{i}{ z}\right)}\right)}{F\left(e^{2\pi i Q\left(\frac{h_Q}{k}+ \frac{iQ}{\ell z}\right)}\right) }  \mathrm{d}z. 
\end{align*}}
Substituting \eqref{bound for F in T1}  in $T_1$ yields
\begin{align}\label{T1}
T_1 = \mathcal{K}_1 + \mathcal{K}_2,
\end{align}
where
{\allowdisplaybreaks
\begin{align}
\mathcal{K}_1  := \sum_{\substack{Q \mid \ell \\ Q^2<\ell}} \left(\frac{Q}{\ell}\right)^{\frac{r}{2}} & \sum_{\substack{0\leq h<k \leq N \\ \gcd(h,k)=1\\ \gcd(k, \ell)=Q}}  \frac{i}{k^2} e^{i\pi r \left(s(h,k) - s\left(\frac{\ell h}{Q},\frac{k}{Q}\right)\right)-2\pi i n \frac{h}{k}}
 \int_{z'_{h,k}}^{z''_{h,k}}  e^{\left(\frac{\pi r}{12z}\left(1-\frac{Q^2}{\ell} \right) +\frac{\pi r z}{12k^2}(\ell-1) +2 \pi n \frac{z}{k^2} \right)} \nonumber \\
& \times \sum_{m,s=0}^{\left\lfloor{\frac{r}{24} \left( \frac{\ell}{Q^2}-1\right)} \right\rfloor} p^{(r)}(m) a_r(s)e^{\frac{2\pi i}{k}( mh' + s Q h_Q)}e^{\frac{-2\pi}{z}\left(m+\frac{Q^2}{\ell} s \right)}  \mathrm{d}z,  \label{Kappa1}
\end{align}}
and  
\begin{align*}
\mathcal{K}_2 &:=  \sum_{\substack{Q \mid \ell \\  Q^2<\ell}} \left(\frac{Q}{\ell}\right)^{\frac{r}{2}} \sum_{\substack{0\leq h<k \leq N \\ \gcd(h,k)=1\\ \gcd(k, \ell)=Q}} \frac{i}{k^2} e^{i\pi r \left(s(h,k) - s\left(\frac{\ell h}{Q},\frac{k}{Q}\right)\right)-2\pi i n \frac{h}{k}} \nonumber \\
&  \times \int_{z'_{h,k}}^{z''_{h,k}}  e^{\left(\frac{\pi r}{12z}\left(1-\frac{Q^2}{\ell} \right) +\frac{\pi r z}{12k^2}(\ell-1) +2 \pi n \frac{z}{k^2} \right)}  \mathcal{O}\left(e^{\frac{-2 \pi Q^2}{\ell} \left(\left\lfloor{\frac{r}{24} \left( \frac{\ell}{Q^2}-1\right)} \right\rfloor +1 \right) \Re\left(\frac{1}{z} \right)} \right) \mathrm{d}z,
\end{align*}
which eventually simplifies as  an error term upon employing Theorem \ref{bounds for generalized kloosterman sum}.  As before,  we consider  the path of integration to be the chord from $z'_{h,k}$ to $z''_{h,k}$ and on this chord, we have $0 \leq \Re(z) \leq \frac{2k^2}{N^2}$, length of the path of integration is less than $\frac{2\sqrt{2}k}{N}$ and $\Re(1/z) \geq 1$. Using these bounds, the error term becomes
\begin{align} \label{final evaluation K2}
\mathcal{K}_2 = \mathcal{O}\left( \frac{1}{N} e^{\frac{4 \pi}{N^2} \left(n+\frac{r}{24}(\ell-1)\right)} \sum_{k=1}^{N}  \frac{1}{k^{\frac{1}{2}-\epsilon}}   \right).
\end{align}
Now, we are left to evaluate the term $\mathcal{K}_1$ \eqref{Kappa1},  which can be rewritten as
\begin{align}\label{form_of_blr(n)_after_condition_on_sums}
\mathcal{K}_1 &= \sum_{\substack{Q \mid \ell \\  Q^2 < \ell}} \left(\frac{Q}{\ell}\right)^{\frac{r}{2}} \sum_{\substack{k=1\\ \gcd(k, \ell)=Q}}^N \frac{i}{k^2} \sum_{m,s=0}^{\left\lfloor{\frac{r}{24} \left( \frac{\ell}{Q^2}-1\right)} \right\rfloor} p^{(r)}(m) a_r(s)C(n) \nonumber \\
& \times \int_{z'_{h,k}}^{z''_{h,k}}  e^{\frac{2\pi}{z}\left(\frac{r}{24}\left(1-\frac{Q^2}{\ell} \right) - \left( m+\frac{Q^2s}{\ell} \right) \right) + \frac{2 \pi z}{k^2} \left( n+ \frac{r}{24}(\ell-1)\right)} \mathrm{d}z,
\end{align}
where 
\begin{align*}
C(n)  := C_{k,r,\ell,  Q,m,s}(n)  =  \sum_{\substack{h~\mathrm{mod}~k \\ \gcd(h,k)=1}} e^{i\pi r \left(s(h,k) - s\left(\frac{\ell h}{Q},\frac{k}{Q}\right)\right)} e^{ \frac{2\pi i}{k}( mh' + s Q h_Q -nh)}. 
\end{align*}
As we make $\Re(z)$ small, so $\Re\left( \frac{1}{z} \right)$ will become large. Thus, for $\frac{r}{24}\left(1-\frac{Q^2}{\ell} \right) \leq \left( m+\frac{Q^2s}{\ell} \right)$, we get an error term whose evaluation is same as we did for $\mathcal{K}_2$. Thus, we leave the details for the reader. Therefore, we are left with the sum runs over all those $m, s$ for which $\frac{r}{24}\left(1-\frac{Q^2}{\ell} \right) > \left( m+\frac{Q^2s}{\ell} \right)$.   Hence,  the sum $\mathcal{K}_1$ becomes
{\allowdisplaybreaks
\begin{align*}
\mathcal{K}_1 = \sum_{\substack{Q \mid \ell \\  Q^2<\ell}} \left(\frac{Q}{\ell}\right)^{\frac{r}{2}} &  \sum_{\substack{k=1\\ \gcd(k, \ell)=Q}}^N \frac{i}{k^2} \sideset{}{'}\sum_{m,s=0}^{\left\lfloor{\frac{r}{24} \left( \frac{\ell}{Q^2}-1\right)} \right\rfloor} p^{(r)}(m) a_r(s) {C}(n) J_{h,k}^* \\
& + \mathcal{O}\left( \frac{1}{N} e^{\frac{4 \pi}{N^2} \left(n+\frac{r}{24}(\ell-1)\right)} \sum_{k=1}^{N}  \frac{1}{k^{\frac{1}{2}-\epsilon}}   \right).
\end{align*}}
where
\begin{align*}
J_{h,k}^*= \int_{z'_{h,k}}^{z''_{h,k}}  e^{\frac{2\pi}{z}\left(\frac{r}{24}\left(1-\frac{Q^2}{\ell} \right) - \left( m+\frac{Q^2s}{\ell} \right) \right) + \frac{2 \pi z}{k^2} \left( n+ \frac{r}{24}(\ell-1)\right)} \mathrm{d}z,
\end{align*}
and $\sum^{'}$ means the sum is running over all those $m, s$ which satisfy $ m+\frac{Q^2s}{\ell} < \frac{r}{24}\left(1-\frac{Q^2}{\ell}\right)$. Now, our aim is to evaluate the integral $J_{h,k}^{*}$ and to do that, we slightly change the our path of integration. We begin by completing the circle $\left|z-\frac{1}{2}\right| = \frac{1}{2}$ and then we subtract the integral over additional arcs, which are the arcs from $0$ to $z'_{h,k}$ and $z''_{h,k}$ to $0$. Thus, we have 
\begin{align*}
J_{h,k}^{*} = \left(\int_{K^{(-)}} -  \int_{0}^{z'_{h,k}} -  \int_{z''_{h,k}}^{0}\right)   e^{\frac{2\pi}{z}\left(\frac{r}{24}\left(1-\frac{Q^2}{\ell} \right) - \left( m+\frac{Q^2s}{\ell} \right) \right) + \frac{2 \pi z}{k^2} \left( n+ \frac{r}{24}(\ell-1)\right)} \mathrm{d}z,
\end{align*}
where $K^{(-)}$ is the circle  $\left|z-\frac{1}{2}\right| = \frac{1}{2}$ in the negative direction. First, we focus on the arc from $0$ to $z'_{h,k}$,  where $\Re\left(\frac{1}{z}\right) = 1$, $\Re(z) \leq \frac{2k^2}{N^2}$, $|z| < |z'_{h,k}| < \frac{k\sqrt{2}}{N}$ and length of this arc is less than $\frac{k\pi}{\sqrt{2}N}$. This gives 
\begin{align*}
&\left|\int_{0}^{z'_{h,k}}  e^{\frac{2\pi}{z}\left(\frac{r}{24}\left(1-\frac{Q^2}{\ell} \right) - \left( m+\frac{Q^2s}{\ell} \right) \right) + \frac{2 \pi z}{k^2} \left( n+ \frac{r}{24}(\ell-1)\right)} \mathrm{d}z \right|  \\
& \leq e^{2\pi\left(\frac{r}{24}\left(1-\frac{Q^2}{\ell} \right) - \left( m+\frac{Q^2s}{\ell} \right) \right) + \frac{4 \pi}{N^2} \left( n+ \frac{r}{24}(\ell-1)\right)} \frac{k\pi}{\sqrt{2}N}.
\end{align*}
Thus,  we have 
\begin{align*}
\int_{0}^{z'_{h,k}}  e^{\frac{2\pi}{z}\left(\frac{r}{24}\left(1-\frac{Q^2}{\ell} \right) - \left( m+\frac{Q^2s}{\ell} \right) \right) + \frac{2 \pi z}{k^2} \left( n+ \frac{r}{24}(\ell-1)\right)} \mathrm{d}z = \mathcal{O}\left(e^{\frac{4 \pi}{N^2} \left( n+ \frac{r}{24}(\ell-1)\right)} \frac{k}{N}\right).
\end{align*}
The similar argument gives the same bound for integral over $z''_{h,k}$ to $0$.   Utilizing these bounds in \eqref{form_of_blr(n)_after_condition_on_sums} and together with Theorem \ref{bounds for generalized kloosterman sum},  we arrive at
\begin{align}\label{final_evaluation_K1}
& \mathcal{K}_1 = \sum_{\substack{Q \mid \ell \\  Q^2<\ell}} \left(\frac{Q}{\ell}\right)^{\frac{r}{2}} \sum_{\substack{k=1\\ \gcd(k, \ell)=Q}}^N \frac{i}{k^2} \sideset{}{'}\sum_{m,s=0}^{\left\lfloor{\frac{r}{24} \left( \frac{\ell}{Q^2}-1\right)} \right\rfloor} p^{(r)}(m) a_r(s) {C}(n) \\
& \times \int_{K^{(-)}} e^{\frac{2\pi}{z}\left(\frac{r}{24}\left(1-\frac{Q^2}{\ell} \right) - \left( m+\frac{Q^2s}{\ell} \right) \right) + \frac{2 \pi z}{k^2} \left( n+ \frac{r}{24}(\ell-1)\right)} \mathrm{d}z + \mathcal{O}\left(\frac{1}{N}e^{\frac{4 \pi}{N^2} \left( n+ \frac{r}{24}(\ell-1)\right)}  \sum_{k=1}^{N} \frac{1}{k^{\frac{1}{2}-\epsilon}}\right).  \nonumber
\end{align}
Finally, substituting \eqref{final evaluation K2} and \eqref{final_evaluation_K1} in \eqref{T1} and combining  with \eqref{final_evaluation_T2} in \eqref{divide blr(n) in two sum}, we get
\begin{align}
& b_\ell^{(r)}(n) = \sum_{\substack{Q \mid \ell \\  Q^2<\ell}} \left(\frac{Q}{\ell}\right)^{\frac{r}{2}} \sum_{\substack{k=1\\ \gcd(k, \ell)=Q}}^{N} \frac{i}{k^2} \sideset{}{'}\sum_{m,s=0} ^{\left\lfloor{\frac{r}{24} \left( \frac{\ell}{Q^2}-1\right)} \right\rfloor} p^{(r)}(m) a_r(s) {C}(n) \label{initial form of b_l_r_n} \\
& \times \int_{K^{(-)}} e^{\frac{2\pi}{z}\left(\frac{r}{24}\left(1-\frac{Q^2}{\ell} \right) - \left( m+\frac{Q^2s}{\ell} \right) \right) + \frac{2 \pi z}{k^2} \left( n+ \frac{r}{24}(\ell-1)\right)} \mathrm{d}z + \mathcal{O}\left(\frac{1}{N}e^{\frac{4 \pi}{N^2} \left( n+ \frac{r}{24}(\ell-1)\right)}  \sum_{k=1}^{N} \frac{1}{k^{\frac{1}{2}-\epsilon}}\right).   \nonumber
\end{align}
Observe that the  left hand side  is independent of $N$ and  the error term on the right hand side goes to zero as $N \rightarrow \infty $, $\sum_{k=1}^{N} \frac{1}{k^s} \sim N^{1-s}$, for $s>0, s\neq 1$.  
Thus,  letting $N \rightarrow \infty$ on both sides of \eqref{initial form of b_l_r_n},  we have 
\begin{align}\label{intermediate_form_of_blr(n)}
b_\ell^{(r)}(n) &= \sum_{\substack{Q \mid \ell \\  Q^2<\ell}} \left(\frac{Q}{\ell}\right)^{\frac{r}{2}} \sum_{\substack{k=1\\ \gcd(k, \ell)=Q}}^{\infty} \frac{i}{k^2} \sideset{}{'}\sum_{m,s=0} ^{\left\lfloor{\frac{r}{24} \left( \frac{\ell}{Q^2}-1\right)} \right\rfloor} p^{(r)}(m) a_r(s) {C}(n) \nonumber \\
& \times \int_{K^{(-)}} e^{\frac{2\pi}{z}\left(\frac{r}{24}\left(1-\frac{Q^2}{\ell} \right) - \left( m+\frac{Q^2s}{\ell} \right) \right) + \frac{2 \pi z}{k^2} \left( n+ \frac{r}{24}(\ell-1)\right)} \mathrm{d}z .
\end{align}
This essentially gives the exact formula for $b_\ell^{(r)}(n)$. However, we need to check the convergence of the above infinite series.  One can see that
\begin{align*}
\left|\int_{K^{(-)}} e^{\frac{2\pi}{z}\left(\frac{r}{24}\left(1-\frac{Q^2}{\ell} \right) - \left( m+\frac{Q^2s}{\ell} \right) \right) + \frac{2 \pi z}{k^2} \left( n+ \frac{r}{24}(\ell-1)\right)} \mathrm{d}z \right| \leq \pi e^{\frac{2 \pi r}{24}\left(1-\frac{Q^2}{\ell} \right) + \frac{2 \pi}{k^2} \left( n+ \frac{r}{24}(\ell-1)\right)}
\end{align*}
as $\Re \left( \frac{1}{z} \right) = 1$ and $\Re(z) \leq 1$ on the circle $K^{(-)}$. 
Using Theorem \ref{bounds for generalized kloosterman sum} and upon simplification,  one can show that 
\begin{align*}
b_\ell^{(r)}(n) = \mathcal{O}\left(  \sum_{k=1}^{\infty}\frac{1}{k^{\frac{3}{2}-\epsilon}} \right),
\end{align*}
which is absolutely convergent for any small $0<\epsilon <1/2. $ The only task left is to write the integral present in \eqref{intermediate_form_of_blr(n)} in a simplified form.  In \eqref{intermediate_form_of_blr(n)},  we substitute $z = \frac{1}{\omega}$. The image of the circle $\left|z-\frac{1}{2}\right| = \frac{1}{2}$  under this transformation is $\Re(\omega) =1$. Thus,  we have
\begin{align*}
b_\ell^{(r)}(n) &= \sum_{\substack{Q \mid \ell \\  Q^2<\ell}} \left(\frac{Q}{\ell}\right)^{\frac{r}{2}} \sum_{\substack{k=1\\ \gcd(k, \ell)=Q}}^\infty \frac{i}{k^2} \sideset{}{'}\sum_{m,s=0}^{\left\lfloor{\frac{r}{24} \left( \frac{\ell}{Q^2}-1\right)} \right\rfloor} p^{(r)}(m) a_r(s)\mathcal{C}(n) \nonumber \\
& \times \int_{1-i\infty}^{1+i\infty} \frac{-1}{\omega^2} e^{2\pi \omega \left(\frac{r}{24}\left(1-\frac{Q^2}{\ell} \right) - \left( m+\frac{Q^2s}{\ell} \right) \right) + \frac{2 \pi }{\omega k^2} \left( n+ \frac{r}{24}(\ell-1)\right)} \mathrm{d}\omega,
\end{align*} 
Substituting $\left(\frac{\pi r}{12}\left(1-\frac{Q^2}{\ell} \right) - 2 \pi\left( m+\frac{Q^2s}{\ell} \right)\right)\omega  = t$ and comparing it with \eqref{bessel_I_integral}, we get 
\begin{align}
b_\ell^{(r)}(n) &= \sum_{\substack{Q \mid \ell \\  Q^2<\ell}} \left(\frac{Q}{\ell}\right)^{\frac{r}{2}} \sum_{\substack{k=1\\ \gcd(k, \ell)=Q}}^\infty \frac{2\pi}{k^2} \sideset{}{'}\sum_{m,s=0}^{\left\lfloor{\frac{r}{24} \left( \frac{\ell}{Q^2}-1\right)} \right\rfloor} p^{(r)}(m) a_r(s) {C}(n)\sqrt{\frac{\frac{\pi r}{12}\left(1-\frac{Q^2}{\ell} \right) - 2 \pi\left( m+\frac{Q^2s}{\ell} \right)}{ \frac{\pi r}{12k^2}(\ell-1) + \frac{2n\pi}{k^2} }} \nonumber \\
& \times I_1\left( \frac{4 \pi}{k}\sqrt{\left(\frac{r}{24}\left(1-\frac{Q^2}{\ell}\right)- \left( m+ \frac{Q^2s}{\ell} \right) \right) \left( n+\frac{r}{24}(\ell-1)\right)}\right). \label{final_b_l_r in terms of I1}
\end{align}
Further,  utilizing 
\begin{align*}
I_{\nu}(z) = \frac{1}{i^{\nu}}J_{\nu}(iz),  \quad
J_1(z) = -\frac{\mathrm{d}}{\mathrm{d}z}J_0(z),
\end{align*}
we get the desired result \eqref{exact formula for r-colored l-regular}.  
\end{proof}

\begin{proof}[Corollary \ref{asymptotic of r-colored l regular partitions theorem}][]
We first separate the main term $M(n)$ corresponding to $Q=1,  k=1,  m=s=0$ in \eqref{final_b_l_r in terms of I1}  and denote the contribution of the remaining terms by $E(n)$. This yields
\begin{align*}
b_\ell^{(r)}(n) 
&= M(n) + E(n),
\end{align*}
where 
\begin{align}\label{main term in asymptotic}
M(n)= \left(\frac{1}{\ell}\right)^{\frac{r}{2}} 2\pi \sqrt{\frac{\frac{ r}{24}\left(1-\frac{1}{\ell} \right)}{ \frac{ r}{24}(\ell-1) + n }} I_1\left( 4 \pi \sqrt{\frac{r}{24}\left(1-\frac{1}{\ell}\right) \left( n+\frac{r}{24}(\ell-1)\right)}\right). 
\end{align}
Next, we employ the classical asymptotic expansion for the modified Bessel function of the first kind,  
\begin{align*}
I_{\nu}(x) \sim \frac{e^{x}}{\sqrt{2\pi x}} \quad \text{as} \quad x \to \infty. 
\end{align*} 
Making use of  this asymptotic expansion into \eqref{main term in asymptotic} gives
\begin{align*}
M(n) &\sim  \frac{1}{\sqrt{2}}\left(\frac{1}{\ell}\right)^{\frac{r}{2}} \left(\frac{ r}{24}\left(1-\frac{1}{\ell} \right) \right)^{\frac{1}{4}} \left( \frac{1}{n + \frac{r}{24}(\ell-1)}\right)^{\frac{3}{4}} e^{ 4 \pi \sqrt{\frac{r}{24}\left(1-\frac{1}{\ell}\right) \left( n+\frac{r}{24}(\ell-1)\right)} }\\
&\sim \frac{1}{\sqrt{2}}\left(\frac{1}{\ell}\right)^{\frac{r}{2}} \left(\frac{ r}{24}\left(1-\frac{1}{\ell} \right) \right)^{\frac{1}{4}} \left( \frac{1}{n}\right)^{\frac{3}{4}} e^{ 4 \pi \sqrt{\frac{nr}{24}\left(1-\frac{1}{\ell}\right)} }
\end{align*} 
as  $n \to \infty$. 
Using the same asymptotic expansion,  one can see that, for any $k \geq 2$,  the modified Bessel function present in $E(n)$ has the following asymptotic expansion: 
\begin{align*}
 I_1\left( \frac{4 \pi}{k}\sqrt{ \delta \left( n+\frac{r}{24}(\ell-1)\right)}\right) \sim  \frac{ e^{\frac{4\pi}{k} \sqrt{ \delta \left( n+\frac{r}{24}(\ell-1)\right)} }}{ \sqrt{ \frac{8 \pi^2}{k} \sqrt{ \delta \left( n+\frac{r}{24}(\ell-1)\right)}}  } \sim   \frac{ e^{\frac{4\pi}{k} \sqrt{ \delta n} }}{ \sqrt{ \frac{8 \pi^2}{k} \sqrt{ \delta  n }}  },  
\end{align*}
where 
\begin{align*}
\delta = \left(\frac{r}{24}\left(1-\frac{Q^2}{\ell}\right)- \left( m+ \frac{Q^2s}{\ell} \right) \right).  
\end{align*}
Moreover,  one can easily check that,  for any $k \geq 2$,  
\begin{align*}
\lim_{n \rightarrow \infty} \frac{ e^{\frac{4\pi}{k} \sqrt{ \delta n} }}{e^{ 4 \pi \sqrt{\frac{nr}{24}\left(1-\frac{1}{\ell}\right)}}} =0.  
\end{align*}
Here we have used the  fact that $\sqrt{\delta} <  \sqrt{\frac{r}{24}\left(1-\frac{1}{\ell}\right)}$ as $(m, s) \neq (0,0)$.  
This completes the proof of Corollary \ref{asymptotic of r-colored l regular partitions theorem}.
\end{proof}

\begin{proof}[Corollary \ref{l regular partition}][]
Substitute $r=1$ in Theorem \ref{l regular r colour } to obtain the desired result.
\end{proof}

\begin{proof}[Corollary \ref{r_colour_distint_partitions}][]
By specializing Theorem \ref{l regular r colour } to the case $\ell = 2$, we arrive at the desired conclusion.
\end{proof}

\begin{proof}[Corollary \ref{sigma_bar_mex_Exact_formula_corollary}][]
By substituting $r=3$ in Corollary \ref{r_colour_distint_partitions}, we get an exact formula for $\sigma\overline{mex}(n)$. Further, we utilize 
\begin{align*}
I_{\nu}(z) = \frac{1}{i^{\nu}}J_{\nu}(iz),  \quad
J_1(z) = -\frac{\mathrm{d}}{\mathrm{d}z}J_0(z),
\end{align*}
to attain the desired form \eqref{sigma_bar_mex_Exact_formula_equation}.
\end{proof}

\begin{proof}[Theorem \ref{hyperbolicity of Jensen polynomial of r-colored l regular partitions}][]

From \eqref{asymptotic of r-colored l regular partitions}, for any non-negative integer $j$, we have
\begin{align*}
\frac{b_\ell^{(r)}(n+j)}{b_\ell^{(r)}(n)} \sim  \left(\frac{n}{n+j} \right) ^{\frac{3}{4}} e^{ 4 \pi \sqrt{\frac{r(n+j)}{24}\left(1- \frac{1}{\ell}\right)} - 4 \pi \sqrt{\frac{r n}{24}\left(1- \frac{1}{\ell}\right)} }.
\end{align*}
Taking natural logarithm on both sides,   we have
\begin{align*}
\log \left(\frac{b_\ell^{(r)}(n+j)}{b_\ell^{(r)}(n)} \right) \sim 4 \pi \sqrt{\frac{r}{24}\left(1- \frac{1}{\ell}\right)}  \left(\sqrt{n+j} - \sqrt{n} \right) - \frac{3}{4} \log \left( \frac{n+j}{n} \right).
\end{align*}
The above right side expression can be simplified as follows:
{\allowdisplaybreaks
\begin{align*}
& 4 \pi \sqrt{\frac{r}{24}\left(1- \frac{1}{\ell}\right)}\left(\sqrt{n+j} - \sqrt{n} \right) - \frac{3}{4} \log \left( \frac{n+j}{n} \right) \\
& = 4 \pi \sqrt{\frac{r}{24}\left(1- \frac{1}{\ell}\right)}  \sum_{i=1}^\infty {1/2 \choose i} \frac{j^i}{n^{i-1/2}} + \frac{3}{4} \sum_{i=1}^\infty \frac{(-1)^i j^i}{in^i}\\
&= \left( 2 \pi \sqrt{\frac{r}{24n}\left(1- \frac{1}{\ell}\right)} - \frac{3}{4n} \right)j - \left( \frac{\pi}{2} \sqrt{\frac{r}{24n^3}\left(1- \frac{1}{\ell}\right)} - \frac{3}{8n^2}  \right)j^2  \\
& + \sum_{i=3}^\infty \left( 4 \pi \sqrt{\frac{r}{24}\left(1- \frac{1}{\ell}\right)} {1/2 \choose i} \frac{1}{n^{i-1/2}} + \frac{3}{4} \frac{(-1)^i}{in^i} \right) j^i.
\end{align*}}
Now, let $A(n)= 2 \pi \sqrt{\frac{r}{24n}\left(1- \frac{1}{\ell}\right)} - \frac{3}{4n}, ~ \delta(n)=\sqrt{\frac{\pi}{2} \sqrt{\frac{r}{24n^3}\left(1- \frac{1}{\ell}\right)} - \frac{3}{8n^2}}$ and 
$$
g_i(n)= 4 \pi \sqrt{\frac{r}{24}\left(1- \frac{1}{\ell}\right)} {1/2 \choose i} \frac{1}{n^{i-1/2}} + \frac{3}{4} \frac{(-1)^i}{in^i},$$
 for all $i\geq 3$. A simple calculation shows that, for $3\leq i \leq d$, we have
$
\lim_{n \rightarrow \infty} \frac{g_i(n)}{(\delta(n))^i} = 0.
$
We can further check that $\lim_{n\rightarrow \infty}\frac{g_i(n)}{(\delta(n))^d} = 0,$ for all $i\geq d+1$.  Hence,  it follows that the sequences $\{b_\ell^{(r)}(n)\},~\{A(n)\},~\{\delta(n)\}$ and $\{g_i(n)\}$ satisfy the assumptions of Theorem \ref{theorem of Griffin ono rolen and zagier}. Therefore, the Jensen polynomials corresponding to the $r$-colored $\ell$-regular partition function $b_\ell^{(r)}(n)$ admit representations in terms of Hermite polynomials, which are hyperbolic for sufficiently large values of $n$. This concludes the proof of Theorem \ref{hyperbolicity of Jensen polynomial of r-colored l regular partitions}.
\end{proof}

\section{Numerical Verification}\label{Numerical Verification}

To demonstrate the accuracy of our exact formula for $b_\ell^{(r)}(n)$, we perform a numerical verification by evaluating explicit values of $b_\ell^{(r)}(n)$ for fixed $\ell = 3$ and $\ell = 6$. These values are computed using the first five terms of the series given in Theorem \ref{l regular r colour } and are then compared with the corresponding exact values.

\begin{table}[h]

\caption{Verification of Theorem \ref{l regular r colour } for $\ell =3$}\label{Table for 3 regular r colour}
\renewcommand{\arraystretch}{1}
{\begin{tabular}{|l|l|l|l|}
\hline
$r$ & $n$ &  Exact value  & Value from our result  \\ 
\hline
 $	50$&  $7$ &  $420621700$ &$420621699.99999$\\      
\hline
$39$&$10$& $24030437457$ & $24030437457$  \\
\hline
$	63$	&$4$ & $849807$&$849806.999999$ \\
\hline
$76$	& $5$&$30050020$ &$30050019.999999$ \\
\hline
%$39$& $4$&   $ 143169$  & $143169.00000$ \\
%\hline
\end{tabular}}
\end{table}

\begin{table}[h]

\caption{Verification of Theorem \ref{l regular r colour } for $\ell =6$}\label{Table for 6 regular r colour}
\renewcommand{\arraystretch}{1}
{\begin{tabular}{|l|l|l|l|}
\hline
$r$ & $n$ &  Exact value  & Value from our result  \\ 
\hline
 $	75$&  $5$ &  $28462590$ &$28462589.9998$\\      
\hline
$25$&$11$& $3755606050$ & $3755606049.9999$  \\
\hline
$	12$	&$22$ & $299225122470$&$299225122470.0000$ \\
\hline
$22$	& $12$&$5175590618$ &$5175590617.9999$ \\
\hline
%$43$& $9$&   $ 9383947276$  & $9383947276.0001$ \\
%\hline
\end{tabular}}
\end{table}

We also present a verification table for the $r$-colored distinct partitions of $n$, denoted by $p_d^{(r)}(n)$, for various values of $r$ and $n$, and compare the computed results with their exact values.

\begin{table}[h]

\caption{Verification of Corollary \ref{r_colour_distint_partitions} }\label{Table for r colour distinct partitions}
\renewcommand{\arraystretch}{1}
{\begin{tabular}{|l|l|l|l|}
\hline
$r$ & $n$ &  Exact value  & Value from our result  \\ 
\hline
$20$&  $5$ &  $46724$ &$46723.99988778628$\\      
\hline
$25$ & $10$ & $206841715$ & $206841714.9999436$  \\
\hline
$30$	&$7$ & $9603210$ & $9603209.99997784$ \\
\hline
$50$	& $9$ & $12141387350$ & $12141387349.99999$ \\
\hline
$73$ & $6$&   $261788950$  & $261788949.9999971$ \\
\hline
\end{tabular}}
\end{table}

Furthermore, we compute the first six terms of the $25$-colored distinct partitions of $10$ numerically by considering the contributions from the six-term series given in Corollary~\ref{r_colour_distint_partitions}.

\begin{align*}
+~206841714&.5985165 \\
+~0&.412498 \\
-~0&.0120247 \\
+~0&.00107527 \\
-~0&.000487608 \\
-~0&.0000162379 \\
\hline
206841714&.9995613
\end{align*}
The exact value of $p_d^{(25)}(10)$ is $206841715$.  Thus,  the error is $\approx 0.0004$.  

\section*{Acknowledgement}
The first author wishes to thank University Grant Commission (UGC), India, for providing Ph.D. scholarship. The second author's research is funded by the Prime Minister Research Fellowship, Govt. of India, Grant No. 2101705. The last author is grateful to the Anusandhan National Research Foundation (ANRF), India, for giving the Core Research Grant CRG/2023/002122 and MATRICS Grant MTR/2022/000545.  We sincerely thank IIT Indore for providing conductive research environment.

\end{document}